\documentclass[a4paper,12pt,reqno]{amsart}

\ExplSyntaxOn
\sys_if_engine_pdftex:T{
\usepackage[T1]{fontenc}
}
\ExplSyntaxOff
\usepackage{microtype}

\usepackage[australian]{babel}
\usepackage{csquotes}

\usepackage{amssymb}
\usepackage{mathtools}
\usepackage{thmtools}

\usepackage{xcolor}

\usepackage[style=ext-alphabetic]{biblatex}

\usepackage{hyperref}

\DeclareFieldFormat{titlecase:title}{\MakeSentenceCase*{#1}}
\DeclareFieldFormat[article]{volume}{\mkbibbold{#1}}
\DeclareFieldFormat[article]{number}{\bibstring{number} #1}
\DeclareDelimFormat{volnumdelim}{\addcomma\space}
\DeclareSourcemap{
  \maps[datatype=bibtex]{
    \map[overwrite]{
      \step[fieldsource=shortjournal,fieldtarget=journaltitle]
    }
  }
}

\hypersetup{
    colorlinks=true,
    citecolor={green!50!black},
    urlcolor=blue,
    hypertexnames=false,
}

\declaretheorem[numberwithin=section]{theorem}
\declaretheorem[numberlike=theorem]{lemma, proposition, corollary, conjecture}
\declaretheorem[style=definition, numberlike=theorem]{definition, example}

\newcommand{\complex}{\mathbb{C}}
\newcommand{\integer}{\mathbb{Z}}
\newcommand{\rational}{\mathbb{Q}}

\newcommand{\aff}{\mathbb{A}}
\newcommand{\proj}{\mathbb{P}}

\newcommand{\monomials}{\mathcal{B}}

\DeclareMathOperator{\kr}{Kr}

\newcommand{\nrat}{N_{\rational}}
\newcommand{\nproj}{N_{\proj}}
\newcommand{\nptimesp}{N_{\proj^1 \times \proj^1}}

\DeclarePairedDelimiter{\abs}{\lvert}{\rvert}
\DeclarePairedDelimiter{\floor}{\lfloor}{\rfloor}
\DeclarePairedDelimiter{\norm}{\lVert}{\rVert}

\newcommand{\closure}[1]{\overline{#1}}

\title{Counting rational points on affine hypersurfaces}
\author{Anders Mah}

\address{School of Mathematics and Statistics, University of New South Wales,
Sydney, NSW 2052, Australia}
\email{anders.mah@unsw.edu.au}



\begin{document}

\begin{abstract}
    We prove an upper bound for the number of rational points of bounded height
    on irreducible affine hypersurfaces.
    More precisely,
    given an irreducible polynomial \(f \in \integer[X_1, \dots, X_n]\),
    we prove an upper bound
    on the number of points \((x_1, \dots, x_n) \in \rational^n\)
    such that \(f(x_1, \dots, x_n) = 0\)
    and each component has height at most \(B\).
    To prove this,
    we require a quantitative form of Hilbert's irreducibility theorem,
    where we bound the number of reducible specialisations
    of an irreducible polynomial at rational points of bounded height.
\end{abstract}

\maketitle

\section{Introduction}

\subsection{Rational points on hypersurfaces}

In this paper we study bounds for the number of solutions to the equation
\begin{equation}\label{eq:f_eq_zero}
    f(x_1, \dots, x_n) = 0,
\end{equation}
where \(f \in \integer[X_1, \dots, X_n]\)
is a polynomial with integer coefficients
and \((x_1, \dots, x_n) \in \rational^n\)
and the height of each variable is bounded.
Recall that the height of a rational number \(t \in \rational\) is defined as
\[
    H(t) \coloneq \max(\abs{x}, \abs{y}),
\]
where \(t = x / y\) is written in lowest terms,
that is, with relatively prime integers \(x\), \(y\).

More precisely,
we study bounds for the following counting function \(\nrat(f, B)\).
\begin{definition}
    For a polynomial \(f \in \integer[X_1, \dots, X_n]\),
    define
    \begin{multline*}
        \nrat(f, B_1, \dots, B_n) \coloneq
        \#\{(x_1, \dots, x_n) \in \rational^n : f(x_1, \dots, x_n) = 0\\
        \text{ and } H(x_i) \leq B_i \text{ for } i = 1, \dots, n\}.
    \end{multline*}
    In the case \(B_1 = \dots = B_n = B\),
    we denote this quantity as simply \(\nrat(f, B)\).
\end{definition}

Equation \eqref{eq:f_eq_zero} defines a hypersurface in affine space,
so a reformulation of the problem
is to bound the number of rational points on affine hypersurfaces.

Many of the results in the literature
give bounds for the number of integral points on affine hypersurfaces,
where the solutions to \eqref{eq:f_eq_zero}
satisfy \((x_1, \dots, x_n) \in \integer^n\).
One such bound is the Bombieri--Pila type bound,
which is a bound of the form
\(O_{d, n, \varepsilon}(B^{n - 2 + 1 / d + \varepsilon})\)
for irreducible polynomials \(f\).
\textcite{bombieriNumberIntegralPoints1989} introduced the determinant method
to prove a bound of this form for \(n = 2\).
Subsequent refinements have extended the bound to all \(n\)
and improved the dependence on \(d\) and \(\varepsilon\),
see, for example,
\cite{binyaminiSharpBoundsNumber2025,%
castryckDimensionGrowthConjecture2020,%
cluckersImprovementsDimensionGrowth2025,%
pilaDensityIntegralRational1995}.
The classical approach to prove Bombieri--Pila type bounds for \(n > 2\)
is to use induction on \(n\).
One considers the intersection with well chosen hyperplanes,
counting the number of points on each intersection separately.
In this approach,
one typically chooses the hyperplanes
so that \(O_{d, n}(B)\) hyperplanes are required to count the points
and that the intersection is irreducible
except for \(O_{d, n}(1)\) hyperplanes.
For the irreducible hyperplane sections, the inductive hypothesis is used,
and for the others, a trivial bound is used.

We also note that the determinant method has been used to bound
the number of rational points on projective hypersurfaces,
where the points are ordered with respect to the projective height.

Let us return to the setting of rational points on affine hypersurfaces.
In this setting,
Bombieri--Pila type bounds are only known when \(n = 2\).
For a non-zero polynomial \(f \in \integer[X_1, \dots, X_n]\),
let \(\norm{f}\) denote the maximum among
absolute values of its coefficients and \(2\).
\textcite[Lemma~2]{heath-brownSquarefreeValues$n^2+1$2012}
proved that for
an irreducible polynomial \(f \in \integer[X_1, X_2]\)
of degree \(d_1\) in \(X_1\) and \(d_2\) in \(X_2\),
one has
\begin{equation}\label{eq:rational_points_two_variable_1}
    \nrat(f, B_1, B_2) \ll_{d_1, d_2, \varepsilon}
    B_1^{2 / d_2 + \varepsilon} \norm{f}^{\varepsilon}.
\end{equation}
We refer to Section~\ref{sec:notation} for notation and conventions.
Under the same conditions for \(f\),
\textcite[Theorem~3.1]{vermeulenDimensionGrowthAffine2024} proved the bound
\begin{equation}\label{eq:rational_points_two_variable_2}
    \nrat(f, B_1, B_2)
    \ll (d_1 d_2)^{7 / 2} B_1^{1 / d_2} B_2^{1 / d_1}.
\end{equation}
Note that these bounds were originally stated as bounds
for the number of rational points on curves in \(\proj^1 \times \proj^1\).
This is equivalent to bounding rational points on curves in \(\aff^2\).

Given that the bounds for integral points can be extended to \(n > 2\),
it is natural to ask if the bounds for rational points
can also be extended to \(n > 2\).
We show that non-trivial bounds are possible under some assumptions.

Our main results are the following bounds for \(\nrat(f, B)\).
\begin{theorem}\label{thm:rational_points_hypersurface}
    Let \(f \in \integer[X_1, \dots, X_n]\)
    be an irreducible polynomial of degree \(d\).
    Suppose that \(f\) has degree at least \(2\) in \(X_1\).
    Then
    \[
        \nrat(f, B) \ll_{d, n} B^{2 n - 3} (\log\norm{f} + \log B)^5.
    \]
\end{theorem}

When \(f\) is monic in \(X_1\),
we prove a Bombieri--Pila type bound for \(n \geq 2\).
\begin{theorem}\label{thm:rational_points_hypersurface_monic}
    Let \(f \in \integer[X_1, \dots, X_n]\)
    be an irreducible polynomial of degree \(d\).
    Let \(d_1\) be the degree in \(X_1\).
    Suppose that \(f\) is monic in \(X_1\).
    Then
    \[
        \nrat(f, B)
        \ll_{d, n} B^{2 (n - 2 + 1 / d_1)} (\log\norm{f} + \log B)^7.
    \]
\end{theorem}

The Schwartz--Zippel bound,
see Lemma~\ref{thm:schwartz_zippel_rational},
gives the trivial bound \(\nrat(f, B) \ll_n d B^{2 n - 2}\)
for non-zero polynomials \(f\).
In Theorem~\ref{thm:rational_points_hypersurface},
the exponent of \(B\) is \(2 n - 3\),
which improves on the exponent \(2 n - 2\) from the Schwartz--Zippel bound
when \(f\) is irreducible and has degree at least \(2\) in \(X_1\).
In Theorem~\ref{thm:rational_points_hypersurface_monic},
the exponent of \(B\) is \(2 n - 4 + 2 / d_1\),
which is an improvement on Theorem~\ref{thm:rational_points_hypersurface}
when \(d_1 \geq 3\).

We also prove a variant of the bound \eqref{eq:rational_points_two_variable_1},
where the \(B_1^{\varepsilon} \norm{f}^{\varepsilon}\) factor
is replaced with \((\log B_2)^2\).
\begin{theorem}\label{thm:rational_points_two_variable}
    Let \(f \in \integer[X_1, X_2]\) be an irreducible polynomial
    of degree \(d_1\) in \(X_1\) and degree \(d_2\) in \(X_2\).
    Then
    \[
        \nrat(f, B_1, B_2)
        \ll_{d_1, d_2} B_1^{2 / d_2} (\log B_2)^2.
    \]
\end{theorem}

Theorem~\ref{thm:rational_points_two_variable}
can be compared with \eqref{eq:rational_points_two_variable_2}.
When \(B_1^{1 / d_2} \asymp B_2^{1 / d_1}\),
the dependence on \(B_1\) and \(B_2\) is slightly better
in \eqref{eq:rational_points_two_variable_2}
than in Theorem~\ref{thm:rational_points_two_variable},
which has an additional factor of \((\log B_2)^2\).
However,
when \(B_2^{1 / d_1} \gg B_1^{1 / d_2} (\log B_1)^2\),
Theorem~\ref{thm:rational_points_two_variable}
gives a stronger bound than \eqref{eq:rational_points_two_variable_2}.

To prove Theorem~\ref{thm:rational_points_hypersurface},
we consider specialisations of \(n - 2\) variables.
Then we consider the cases where the specialisation
(which is a bivariate polynomial)
is irreducible, reducible and zero,
and bound the contribution to \(\nrat(f, B)\)
from each cases separately.
We use Theorem~\ref{thm:rational_points_two_variable}
to bound the number of points coming from the irreducible case.
We use the Schwartz--Zippel bound, see Lemma~\ref{thm:schwartz_zippel_rational},
to bound the number of points coming from the reducible case.
To bound the number of times the reducible case occurs,
we use Corollary~\ref{thm:hilbert_irreducibility},
which is a quantitative form of Hilbert's irreducibility theorem.

We use a similar approach
to prove Theorem~\ref{thm:rational_points_hypersurface_monic},
but we take more care when bounding the contribution
from the reducible case.

The approach used in Theorems~\ref{thm:rational_points_hypersurface}
and \ref{thm:rational_points_hypersurface_monic}
differ somewhat from the classical approach
of \textcite{pilaDensityIntegralRational1995}
that is used to count integral points on affine hypersurfaces.
\citeauthor{pilaDensityIntegralRational1995}
considers the intersection with hyperplanes of the form
\[
    a_1 X_1 + \dots + a_n X_n = b,
\]
where the \(a_i\) are chosen so that the intersection is irreducible
except for \(O_{d, n}(1)\) values of \(b\).
When counting integral points of height at most \(B\),
it suffices to consider values of \(b\) with absolute value \(O_{d, n}(B)\).
This approach seems to be difficult to apply
to Theorems~\ref{thm:rational_points_hypersurface}
and \ref{thm:rational_points_hypersurface_monic}
because it is more difficult to bound the number of \(b\)
where the intersection contributes to \(\nrat(f, B)\).

To prove Theorem~\ref{thm:rational_points_two_variable},
we modify the proof of \cite[Theorem~3.1]{vermeulenDimensionGrowthAffine2024}.
The approach is to use the determinant method
to construct and bound the degree of an auxiliary polynomial \(g\),
which is not divisible by \(f\)
and vanishes at the zeros of \(f\) of height at most \((B_1, B_2)\).

\subsection{Quantitative Hilbert's irreducibility theorem}

The classical Hilbert's irreducibility theorem
states that for an irreducible polynomial
\(f \in \integer[T_1, \dots, T_s, Y_1, \dots, Y_r]\),
there are infinitely many \((t_1, \dots, t_s) \in \rational^s\)
such that the specialisation
\[
    f(t_1, \dots, t_s, Y_1, \dots, Y_r) \in \rational[Y_1, \dots, Y_r]
\]
is also irreducible, see, for example,
\cite[Section~9.6]{bombieriHeightsDiophantineGeometry2006}
or \cite[Chapter~13]{friedFieldArithmetic2023}.

The quantitative forms of Hilbert's irreducibility theorem
give bounds for the number of \((t_1, \dots, t_s)\)
such that the specialisation is reducible.
Many of these results prove upper bounds
for the number of \((t_1, \dots, t_s) \in \integer^s\)
such that \(\abs{t_i} \leq B\) for \(i = 1, \dots, s\)
and \(f(t_1, \dots, t_s, Y_1, \dots, Y_r)\)
is reducible in \(\rational[Y_1, \dots, Y_r]\),
see, for example,
\cite{castilloHilbertsIrreducibilityTheorem2017,%
cluckersImprovementsDimensionGrowth2025,%
cohenDistributionGaloisGroups1981,%
debesBoundsHilbertsIrreducibility2008,%
paredesEffectiveHilbertsIrreducibility2024,%
walkowiakTheoremeDirreductibiliteHilbert2005,%
zywinaHilbertsIrreducibilityTheorem2010}.

In order to prove the main results for \(\nrat(f, B)\),
we need a quantitative form of Hilbert's irreducibility theorem
where \((t_1, \dots, t_s) \in \rational^s\).
More precisely,
we need to bound the following quantity.
\begin{definition}
    For \(f \in \integer[T_1, \dots, T_s, Y_1, \dots, Y_r]\),
    define
    \begin{multline*}
        S(f, B_1, \dots, B_s) \coloneq \#\{(t_1, \dots, t_s) \in \rational^s :\\
        f(t_1, \dots t_s, Y_1, \dots, Y_r) \text{ is reducible in }
        \rational[Y_1, \dots, Y_r]\\
        \text{ and } H(t_i) \leq B_i \text{ for } i = 1, \dots, s\}.
    \end{multline*}
    In the case when \(B_1 = \dots = B_s = B\),
    we denote this quantity as simply \(S_s(f, B)\).
\end{definition}
\begin{theorem}\label{thm:hilbert_irreducibility_non_uniform}
    Let \(f \in \integer[T_1, \dots, T_s, Y_1, \dots, Y_r]\)
    be an irreducible polynomial of degree \(d\).
    Assume that \(B_1 \leq B_2 \leq \dots \leq B_s\).
    Then
    \[
        S(f, B_1, \dots, B_s)
        \ll_{d, r, s} B_1 \left(\prod_{i = 2}^s B_i^2\right)
        (\log\norm{f} + \log B_s)^5.
    \]
\end{theorem}
We use the following corollary to bound \(\nrat(f, B)\),
which follows from Theorem~\ref{thm:hilbert_irreducibility_non_uniform}
after taking \(B = B_1 = \dots = B_s\).
\begin{corollary}\label{thm:hilbert_irreducibility}
    Let \(f \in \integer[T_1, \dots, T_s, Y_1, \dots, Y_r]\)
    be an irreducible polynomial of degree \(d\).
    Then,
    for any \(B \geq 2\),
    \[
        S_s(f, B) \ll_{d, r, s} B^{2 s - 1} (\log\norm{f} + \log B)^5.
    \]
\end{corollary}

The strategy of the proof of
Theorem~\ref{thm:hilbert_irreducibility_non_uniform} is as follows.
We first consider the case \(r = s = 1\),
for which we adapt the approach of
\cite{schinzelLeastAdmissibleValue1995,
walkowiakTheoremeDirreductibiliteHilbert2005,
paredesEffectiveHilbertsIrreducibility2024}.
We reduce the problem
to the problem of bounding the number of specialisations of certain polynomials
that have a rational zero.
Then we use a Bombieri--Pila type bound for projective curves
to bound the number of these specialisations.

For the general case,
we adapt the proof of
\cite[Theorem~1.10]{cluckersImprovementsDimensionGrowth2025}.
We reduce the \(s = 1\) case to the \(r = s = 1\) case
using a Kronecker substitution.
Finally,
we prove the general case by induction on \(s\).

\subsection{Structure of the paper}

In Section~\ref{sec:hilbert_irreducibility},
we prove Theorem~\ref{thm:hilbert_irreducibility_non_uniform}.
In Section~\ref{sec:rational_points_curves},
we prove Theorem~\ref{thm:rational_points_two_variable}
assuming that the auxiliary polynomial has been constructed.
In Section~\ref{sec:auxiliary_polynomial},
we construct the auxiliary polynomial using the determinant method.
In Section~\ref{sec:rational_points_hypersurface},
we prove Theorems~\ref{thm:rational_points_hypersurface}
and \ref{thm:rational_points_hypersurface_monic}.
In Section~\ref{sec:future_research_directions},
we state some possible directions for future research.

\subsection{Notation and conventions}
\label{sec:notation}

The notations \(U = O_{\rho}(V)\), \(V = \Omega_{\rho}(U)\),
\(U \ll_{\rho} V\) and \(V \gg_{\rho} U\)
are equivalent to \(\abs{U} \leq c V\) for some positive constant \(c\),
which may depend on the parameters included in \(\rho\)
(thus these symbols used without subscripts mean that the implied constants
are absolute).

For a polynomial \(f \in \integer[X_1, \dots, X_n]\),
we say that \(f\) is irreducible if it is irreducible over \(\rational\),
unless otherwise specified.
We say that \(f\) is primitive
if the greatest common divisor of its coefficients is one.

\subsection{Acknowledgements}

I would like to thank my supervisor,
Igor Shparlinski,
for introducing this topic to me and for helpful advice.

\section{Hilbert's irreducibility theorem}
\label{sec:hilbert_irreducibility}

The goal of this section is to prove
Theorem~\ref{thm:hilbert_irreducibility_non_uniform}.

\subsection{Proof outline for the case \texorpdfstring{\(s = r = 1\)}{s = r = 1}}

We start by establishing Theorem~\ref{thm:hilbert_irreducibility_non_uniform}
in the case \(s = r = 1\),
see Proposition~\ref{thm:hilbert_irreducibility_two_variable}.
We construct polynomials \(P_{\omega} \in \integer[T, Y]\)
such that if \(f(t, Y)\) is reducible,
then some \(P_{\omega}(t, Y)\) has a rational root.

Next, we count the number of \(t \in \rational\) with \(H(t) \leq B\)
such that \(P_{\omega}(t, Y)\) has a rational root.
To do this,
we bound the height of any \(y \in \rational\)
satisfying \(P_{\omega}(t, y) = 0\).
We define a projective curve
such that solutions of \(P_{\omega}(t, y) = 0\)
give rational points on the curve.
Then we apply the bound of
\cite[Corollary~3.1.2]{castryckDimensionGrowthConjecture2020},
on the number of rational points on projective curves.
This bounds the number of \(t\) such that \(P_{\omega}(t, Y)\)
has a rational root,
which implies a bound for the number of \(t\)
such that \(f(t, Y)\) is reducible.

\subsection{Construction of the polynomials}

Let \(f \in \integer[T, Y]\) be an irreducible polynomial
of degree \(d_T\) in \(T\) and degree \(d_Y\) in \(Y\).
Let
\[
    f(T, Y) = a_0(T) \prod_{i = 1}^{d_Y} (Y - y_i(T))
\]
be the factorisation of \(f\) in \(\closure{\rational(T)}[Y]\),
where \(y_i(T) \in \closure{\rational(T)}\).
Let \(\omega \subseteq \{1, \dots, d_Y\}\).
The coefficients of \(a_0(T) \prod_{i \in \omega} (Y - y_i(T))\)
as a polynomial in \(Y\)
are of the form \(\pm a_0(T) \tau_j(y_i(T) : i \in \omega)\),
where \(\tau_j\) is the \(j\)th elementary symmetric function of its arguments.

Let \(P_{\omega, j}(T, Y)\) be the minimal polynomial of
\(a_0(T) \tau_j(y_i(T) : i \in \omega)\) over \(\rational(T)\).
By the remark in \cite[296]{schinzelLeastAdmissibleValue1995},
\(P_{\omega, j}\) is a polynomial in \(\integer[T, Y]\)
and it is monic in \(Y\).
Since \(f(T, Y)\) is irreducible over \(\rational(T)\),
at least one coefficient of \(a_0(T) \prod_{i \in \omega} (Y - y_i(T))\),
treated as a polynomial in \(Y\),
is not an element of \(\rational(T)\).
Therefore \(\deg_Y(P_{\omega, j}) \geq 2\) for some \(j\).
Denote this \(P_{\omega, j}\) by \(P_{\omega}\).

\begin{lemma}\label{thm:P_omega}
    Let \(f \in \integer[T, Y]\) be an irreducible polynomial
    of degree \(d_T\) in \(T\) and degree \(d_Y \geq 2\) in \(Y\).
    Let \(t \in \rational\) such that \(f(t, Y)\)
    is reducible over \(\rational\).
    Suppose that \(a_0(t) \neq 0\).
    Then there exists a non-empty subset \(\omega \subseteq \{1, \dots, d_Y\}\)
    of cardinality \(\abs{\omega} \leq d_Y / 2\)
    such that \(P_{\omega}(t, Y)\) has a rational root
    and we have the inequalities
    \begin{equation}\label{eq:P_omega_deg}
        \deg(P_{\omega})
        \leq d_T \deg_Y(P_{\omega})
        \leq d_T \binom{d_Y}{\abs{\omega}}
        \leq d_T 2^{d_Y}
    \end{equation}
    and
    \begin{equation}\label{eq:P_omega_norm}
        \norm{P_{\omega}} \leq (2^{d_Y + 1} (d_T + 1)
        \norm{f})^{\deg_Y(P_{\omega})}.
    \end{equation}
\end{lemma}
\begin{proof}
    Each \(a_0(T) y_i(T)\) is integral over \(\rational[T]\)
    by the argument before the lemma.
    By \cite[Chapter~5, Exercise~2]{atiyahIntroductionCommutativeAlgebra1969},
    the specialisation map \(T \mapsto t\) defined on \(\rational[T]\)
    can be extended to \(a_0(T) y_i(T)\) for \(i = 1, \dots, d_Y\).
    Furthermore, since \(a_0(t) \neq 0\),
    the specialisation map can be extended to \(y_i(T)\).
    We denote the specialisation of \(y_i(T)\) by \(y_i(t)\).

    Since \(f(t, Y)\) is reducible in \(\rational[Y]\),
    we can write
    \[
        f(t, Y) = a_0(t) \prod_{i \in \omega} (Y - y_i(t))
        \prod_{i \notin \omega} (Y - y_i(t)),
    \]
    where \(a_0(t) \prod_{i \in \omega} (Y - y_i(t)) \in \rational[Y]\),
    for some \(\omega \subseteq \{1, \dots, d_Y\}\)
    with \(\abs{\omega} \leq d_Y / 2\).
    Then \(a_0(t) \tau_j(y_i(t) : i \in \omega) \in \rational\)
    because it is a coefficient of a rational polynomial.
    Since \(P_{\omega, j}\) is the minimal polynomial
    of \(a_0(T) \tau_j(y_i(T) : i \in \omega)\),
    we have
    \[
        P_{\omega, j}(t, a_0(t) \tau_j(y_i(t) : i \in \omega)) = 0.
    \]
    In particular,
    \(P_{\omega}(t, Y)\) has a rational root.

    The inequalities \eqref{eq:P_omega_deg}
    and \eqref{eq:P_omega_norm} are proved in
    \cite[Lemme~3.1]{walkowiakTheoremeDirreductibiliteHilbert2005}.
\end{proof}

Lemma~\ref{thm:P_omega}
lets us bound the number of \(t\) such that \(f(t, Y)\) is reducible
by bounding the number of \(t\) such that some \(P_{\omega}(t, Y)\)
has a rational root.
Let us look at an example.
\begin{example}
    Consider the polynomial
    \[
        f(T, Y) = Y^4 + 2 Y^2 + 3 T^2.
    \]
    Note that \(f\) is irreducible in \(\rational[T, Y]\).
    We can factor \(f\) over \(\closure{\rational(T)}\) as
    \[
        f(T, Y) = (Y - y_1(T)) (Y - y_2(T)) (Y - y_3(T)) (Y - y_4(T)).
    \]
    where
    \begin{align*}
        y_1(T) = -y_2(T) &= \sqrt{-1 + \sqrt{1 - 3 T^2}},\\
        y_3(T) = -y_4(T) &= \sqrt{-1 - \sqrt{1 - 3 T^2}}.
    \end{align*}
    Let \(\omega = \{1, 2\}\).
    Then \(P_{\omega}(T, Y)\) is the minimal polynomial of \(y_1(T) y_2(T)\)
    over \(\rational(T)\), given by
    \[
        P_{\omega}(T, Y) = Y^2 - 2 Y + 3 T^2.
    \]
    Observe that
    \[
        f\left(\frac{1}{2}, Y\right)
        = Y^4 + 2 Y^2 + \frac{3}{4}
        = \left(Y^2 + \frac{1}{2}\right) \left(Y^2 + \frac{3}{2}\right)
    \]
    is reducible and that \(y_1(1 / 2) y_2(1 / 2) = 1 / 2\)
    is a rational root of \(P_{\omega}(1 / 2, Y)\).
\end{example}

\subsection{Counting specialisations with rational roots}

In order to bound the number of \(t\) such that \(P_{\omega}(t, Y)\)
has a rational root,
we first show that these \(t\) give rational points of bounded height
on the projective curve defined by the homogenisation of \(P_{\omega}\).
Then we bound the number of rational points of bounded height
on this projective curve.

The following result is certainly well known,
see, for example, \cite{fujiwaraUeberObereSchranke1916}.
It also follows instantly
from \cite[Theorem~4.2~(iii)]{mignotteMathematicsComputerAlgebra1992}
taken with
\[
    \lambda_k = 2^k \frac{2^d - 1}{2^d}, \quad k=1, \dots, n.
\]

\begin{lemma}[Lagrange]\label{thm:complex_root_bound}
    Let \(f(X) = X^d + a_{d - 1} X^{d - 1} + \dots + a_0 \in \complex[X]\).
    Then any root \(\alpha\) of \(f\) satisfies
    \[
        \abs{\alpha} \leq 2 \max_{1 \leq k \leq d}
        \abs{a_{d - k}}^{1 / k}.
    \]
\end{lemma}

\begin{lemma}\label{thm:f_t_rational_root_bound}
    Let \(f \in \integer[T, Y]\) be a polynomial,
    monic in \(Y\),
    of degree \(d_T\) in \(T\) and degree \(d_Y\) in \(Y\).
    Let \(t = a / b \in \rational\) be written in lowest terms.
    Let \(y \in \rational\) be such that \(f(t, y) = 0\).
    Then we can write
    \(y = c / b^{d_T}\) for some \(c \in \integer\)
    such that \(\abs{c} \leq 2 (d_T + 1) \norm{f} H(t)^{d_T}\).
\end{lemma}
\begin{proof}
    Define the polynomial \(g\) by
    \[
        g(T, Y) \coloneq b^{d_T d_Y} f\left(T, \frac{Y}{b^{d_T}}\right).
    \]
    Note that \(y\) is a root of \(f(t, Y)\) if and only if
    \(b^{d_T} y\) is a root of \(g(t, Y)\).
    Therefore it suffices to show that every rational root \(c\)
    of \(g(t, Y)\) is an integer satisfying
    \(\abs{c} \leq 2 (d_T + 1) \norm{f} H(t)^{d_T}\).

    Now write
    \[
        f(T, Y) = Y^{d_Y} + \sum_{i = 0}^{d_Y - 1} f_i(T) Y^i,
    \]
    where \(f_i \in \integer[T]\) are polynomials
    of degree of at most \(d_T\).
    Then
    \[
        g(t, Y) = Y^{d_Y} + \sum_{i = 0}^{d_Y - 1}
        b^{(d_Y - i) d_T} f_i(t) Y^i.
    \]
    Each \(b^{d_T} f_i(t)\) is a sum of the form
    \begin{equation}\label{eq:g_coefficient_form}
        f_{i, 0} b^{d_T} + f_{i, 1} b^{d_T - 1} a + \dots + f_{i, d_T} a^{d_T},
    \end{equation}
    which is an integer.
    Hence \(g(t, Y)\) is a monic polynomial with integer coefficients,
    so any rational root must be an integer.

    We claim that
    \begin{equation}\label{eq:g_coefficient_bound}
        \abs{b^{(d_Y - i) d_T} f_i(t)}^{1 / (d_Y - i)}
        \leq (d_T + 1) \norm{f} H(t)^{d_T}
    \end{equation}
    for all \(0 \leq i \leq d_Y - 1\).
    If \(\abs{f_i(t)} < 1\),
    then
    \[
        \abs{b^{(d_Y - i) d_T} f_i(t)}^{1 / (d_Y - i)}
        \leq \abs{b}^{d_T}
        \leq (d_T + 1) \norm{f} H(t)^{d_T}.
    \]
    If \(\abs{f_i(t)} \geq 1\),
    then
    \[
        \abs{b^{(d_Y - i) d_T} f_i(t)}^{1 / (d_Y - i)}
        \leq \abs{b^{d_T} f_i(t)}.
    \]
    We have \(\abs{f_{i, j}} \leq \norm{f}\) for all \(i\), \(j\)
    and \(\abs{a}, \abs{b} \leq H(t)\).
    Putting this into \eqref{eq:g_coefficient_form} gives
    \[
        \abs{b^{d_T} f_i(t)} \leq (d_T + 1) \norm{f} H(t)^{d_T}.
    \]
    This proves that \eqref{eq:g_coefficient_bound} holds in all cases.

    We apply Lemma~\ref{thm:complex_root_bound} to \(g(t, Y)\)
    and use the bound \eqref{eq:g_coefficient_bound}
    to get that any rational root \(c\) of \(g(t, Y)\)
    must satisfy \(\abs{c} \leq 2 (d_T + 1) \norm{f} H(t)^{d_T}\).
\end{proof}

To state the bound on the number of rational points of bounded height
on projective curves,
we first define a counting function.
Recall that the height of \((x_0 : \dots : x_n) \in \proj^n\),
where \(x_0, \dots, x_n \in \integer\) and
\(\gcd(x_0, \dots, x_n) = 1\) is defined as
\[
    H(x_0 : \dots : x_n) \coloneq \max(\abs{x_0}, \dots, \abs{x_n}).
\]
\begin{definition}
    For a homogeneous polynomial \(f \in \integer[X_0, \dots, X_n]\),
    define
    \begin{multline*}
        \nproj(f, B) \coloneq
        \#\{(x_0 : \dots : x_n) \in \proj^n : f(x_0, \dots, x_n) = 0\\
        \text{ and } H(x_0 : \dots : x_n) \leq B\}.
    \end{multline*}
\end{definition}
We use the following bound on \(\nproj(f, B)\).
\begin{lemma}[{\cite[Corollary~3.1.2]{castryckDimensionGrowthConjecture2020}}]
    \label{thm:projective_curve_bound}
    Let \(f \in \integer[U, T, Y]\) be an irreducible homogeneous polynomial
    of degree \(d\).
    Then
    \[
        \nproj(f, B) \ll d^4 B^{2 / d}.
    \]
\end{lemma}

The following lemma gives a bound for the number of specialisations
of a polynomial with a rational root,
which we apply to the \(P_{\omega}\).
\begin{lemma}\label{thm:number_rational_specialisation_with_zero}
    Let \(f \in \integer[T, Y]\) be an irreducible polynomial,
    monic in \(Y\),
    of degree \(d_T\) in \(T\) and degree \(d_Y\) in \(Y\).
    The number of \(t \in \rational\) with \(H(t) \leq B\)
    such that \(f(t, Y)\) has a root in \(\rational\) is
    \[
        O(d_Y^8 d_T^4 (\log \norm{f})^5 B^{2 / d_Y}).
    \]
\end{lemma}
\begin{proof}
    If \(d_Y = 1\),
    then the lemma is true because there are \(O(B^2)\)
    values of \(t \in \rational\) such that \(H(t) \leq B\),
    so we may assume that \(d_Y \geq 2\).

    Let \(t \in \rational\) with \(H(t) \leq B\)
    and let \(y \in \rational\) be a root of \(f(t, Y)\).
    Write \(t = a / b\) in lowest terms.
    By Lemma~\ref{thm:f_t_rational_root_bound},
    we can write \(y = c / b^{d_T}\) where \(c \in \integer\)
    and \(\abs{c} \leq 2 (d_T + 1) \norm{f} B^{d_T}\).

    Let \(d\) be the total degree of \(f\).
    We split into cases based on whether \(d\) is large or small.
    Let \(H \coloneq \max(e^e, \norm{f})\).
    Let \(L_1 \coloneq \log H\) and \(L_2 \coloneq \log\log H\).

    \textbf{Case 1:} \(d \geq d_T d_Y L_1 / L_2\).
    Consider the homogenisation
    \[
        g(U, T, Y) \coloneq U^d f\left(\frac{T}{U}, \frac{Y}{U}\right).
    \]
    The zero \((t, y)\) of \(f\)
    gives an integral zero \((b^{d_T}, a b^{d_T - 1}, c)\) of \(g\).
    Furthermore,
    these zeros give distinct projective points.
    Also note that \(H(b^{d_T} : a b^{d_T - 1} : c)
    \leq 2 (d_T + 1) \norm{f} B^{d_T}\),
    so it suffices to show that
    \[
        \nproj(g, 2 (d_T + 1) \norm{f} B^{d_T})
        \ll d_Y^8 d_T^4 (\log \norm{f})^5 B^{2 / d_Y}.
    \]

    We know that \(g\) is an irreducible polynomial of degree \(d\),
    because it is the homogenisation of the irreducible polynomial \(f\),
    so Lemma~\ref{thm:projective_curve_bound} gives the bound
    \begin{equation}\label{eq:g_projective_bound}
        \nproj(g, 2 (d_T + 1) \norm{f} B^{d_T})
        \ll d^4 (d_T + 1)^{2 / d} \norm{f}^{2 / d} B^{2 d_T / d}.
    \end{equation}
    Since \(d \geq d_T d_Y L_1 / L_2\) and \(d_Y \geq 2\),
    we have \((d_T + 1)^{2 / d} \leq 2\)
    and \(\norm{f}^{2 / d} \leq H^{L_2 / L_1} \leq \log H\)
    and \(B^{2 d_T / d} \leq B^{2 / d_Y}\).
    Putting this into \eqref{eq:g_projective_bound} gives
    \[
        \nproj(g, (d_T + 1) \norm{f} B^{d_T})
        \ll d^4 (\log \norm{f}) B^{2 / d_Y}.
    \]
    This gives the desired bound since \(d^4 \ll d_Y^8 d_T^4\).

    \textbf{Case 2:} \(d < d_T d_Y L_1 / L_2\).
    Let \(E = \floor{d_T d_Y L_1 / L_2} + 1\).
    Consider the polynomial \(\widetilde{f}(T, Y) \coloneq f(T, T^E + Y)\)
    and let
    \[
        \widetilde{g}(U, T, Y)
        \coloneq U^{\deg\widetilde{f}}
        \widetilde{f}\left(\frac{T}{U}, \frac{Y}{U}\right)
    \]
    be the homogenisation of \(\widetilde{f}\).
    Note that \(\widetilde{f}\) is irreducible because \(f\) is irreducible
    and every factor \(h(T, Y)\) of \(\widetilde{f}\)
    gives a factor \(h(T, Y - T^E)\) of \(f\).
    We also have
    \[
        d_Y E \leq \deg\widetilde{g} \leq d_Y E + d_T \leq 3 d_Y^2 d_T L_1.
    \]
    Every zero \((t, y)\) of \(\widetilde{f}\)
    corresponds to a zero \((t, t^E + y)\) of \(f\),
    so it suffices to bound the number of \(t\)
    such that \(\widetilde{f}(t, Y)\) has a rational root.
    By applying Lemma~\ref{thm:f_t_rational_root_bound}
    to the zero \((t, t^E + y)\) of \(f\),
    we can write \(t^E + y = c / b^{d_T}\)
    where \(\abs{c} \leq 2 (d_T + 1) \norm{f} B^{d_T}\).
    Then
    \[
        y
        = (t^E + y) - t^E
        = \frac{c}{b^{d_T}} - \frac{a^E}{b^E}
        = \frac{c b^{E - d_T} - a^E}{b^E}.
    \]
    Note that
    \begin{align*}
        \abs{c b^{E - d_T} - a^E}
        &\leq \abs{c} \abs{b}^{E - d_T} + \abs{a}^E\\
        &\leq 2 (d_T + 1) \norm{f} B^E + B^E\\
        &\leq 3 (d_T + 1) \norm{f} B^E.
    \end{align*}
    Therefore a rational root of \(\widetilde{f}(t, Y)\)
    gives a zero \((b^E, a b^{E - 1}, c b^{E - d_T} - a^E)\)
    of \(\widetilde{g}\) with height at most \(3 (d_T + 1) \norm{f} B^E\).
    To bound such zeros,
    we apply Lemma~\ref{thm:projective_curve_bound},
    similarly to case 1,
    which gives
    \begin{align*}
        \nproj(\widetilde{g}, 3 (d_T + 1) \norm{f} B^E)
        &\ll (\deg \widetilde{g})^4 ((d_T + 1)
        \norm{f} B^E)^{2 / \deg \widetilde{g}}\\
        &\ll d_Y^8 d_T^4 (\log \norm{f})^5 B^{2 / d_Y}.
    \end{align*}
    This completes the proof.
\end{proof}

\subsection{Univariate specialisation}

We first prove Theorem~\ref{thm:hilbert_irreducibility_non_uniform}
in the special case of \(s = r = 1\).
\begin{proposition}\label{thm:hilbert_irreducibility_two_variable}
    Let \(f \in \integer[T, Y]\) be irreducible over \(\rational\)
    of degree \(d\).
    Then
    \[
        S_1(f, B) \ll_d (\log \norm{f})^5 B.
    \]
\end{proposition}
\begin{proof}
    For \(\omega \subseteq \{1, \dots, d_Y\}\),
    let \(S_{\omega}(B)\) be the number of \(t\) with \(H(t) \leq B\)
    such that \(P_{\omega}(t, Y)\) has a rational root.
    By Lemmas~\ref{thm:P_omega}
    and \ref{thm:number_rational_specialisation_with_zero},
    we have
    \begin{align*}
        S_{\omega}(B)
        &\ll (\deg_Y P_{\omega})^8 (\deg_T P_{\omega})^4
        (\log \norm{P_{\omega}})^5 B^{2 / \deg_Y P_{\omega}}\\
        &\ll 2^{12 d_Y} d_T^4 (2^{d_Y} ((d_Y + 1) \log 2 + \log(d_T + 1)
        + \log \norm{f}))^5 B\\
        &\ll 2^{17 d_Y} d_Y^5 d_T^4 (\log(d_T + 1) + \log \norm{f})^5 B.
    \end{align*}
    There are at most \(d_T\) values of \(t\) such that \(a_0(t) = 0\).
    By Lemma~\ref{thm:P_omega},
    if \(t\) is such that \(a_0(t) \neq 0\) and \(f(t, Y)\) is reducible
    then \(t\) must be counted by some \(S_{\omega}(B)\).
    Therefore
    \[
        S_1(f, B)
        \leq d_T + \sum_{\omega \subseteq \{1, \dots, d_Y\}} S_{\omega}(B)
        \leq d_T + 2^{d_Y} \max_{\omega \subseteq \{1, \dots, d_Y\}}
        S_{\omega}(B).
    \]
    We conclude that
    \begin{align*}
        S_1(f, B)
        &\ll 2^{18 d_Y} d_Y^5 d_T^4 (\log(d_T + 1) + \log \norm{f})^5 B\\
        &\ll_d (\log \norm{f})^5 B,
    \end{align*}
    as desired.
\end{proof}

The exponent of \(B\)
in Proposition~\ref{thm:hilbert_irreducibility_two_variable}
is optimal as the following example shows.
\begin{example}
    Consider the polynomial
    \[
        f(T, Y) = Y^2 - T.
    \]
    Then \(f\) is irreducible in \(\rational[T, Y]\).
    Let \(a, b \in \integer\) with \(0 < a, b \leq \sqrt{B}\)
    and \(\gcd(a, b) = 1\).
    If \(t = a^2 / b^2\),
    then \(H(t) = \max(a^2, b^2) \leq B\)
    and we can factor \(f(t, Y)\) in \(\rational[Y]\) as
    \[
        f(t, Y) = Y^2 - \frac{a^2}{b^2}
        = \left(Y + \frac{a}{b}\right) \left(Y - \frac{a}{b}\right).
    \]
    This gives the bound
    \[
        S_1(f, B) \gg B.
    \]
\end{example}

We now prove Theorem~\ref{thm:hilbert_irreducibility_non_uniform}
for \(s = 1\) and \(r \geq 1\).
\begin{lemma}\label{thm:hilbert_irreducibility_r}
    Let \(f \in \integer[T, Y_1, \dots, Y_r]\) be an irreducible polynomial
    of degree \(d\).
    Then
    \[
        S_1(f, B) \ll_{d, r} (\log \norm{f})^5 B.
    \]
\end{lemma}
\begin{proof}
    We adapt the proof of
    \cite[Theorem~1.10]{cluckersImprovementsDimensionGrowth2025}.
    Let \(a \coloneq 1 + \max_{1 \leq i \leq r} \deg_{Y_i} f\).
    Consider the Kronecker substitution
    \[
        \kr(f)(T, Y) \coloneq f(T, Y, Y^a, Y^{a^2}, \dots, Y^{a^{r - 1}})
        \in \integer[T, Y].
    \]
    Notice that \(\norm{\kr(f)} = \norm{f}\)
    and \(\deg(\kr(f)) \leq a^{r - 1} d\).
    Let
    \[
        \kr(f)(T, Y) = \prod_{i = 1}^m Q_i(T, Y)
    \]
    be the factorisation of \(\kr(f)\) into irreducible polynomials
    in \(\integer[T, Y]\).
    Then \(\deg(Q_i) \leq \deg(\kr(f)) \leq a^{r - 1} d\).
    By
    \cite[Chapter~3, Proposition~2.3]{langFundamentalsDiophantineGeometry1983},
    we have
    \[
        \norm{Q_i} \leq \norm{Q_i} \norm{\kr(f) / Q_i}
        \leq 4^{(\deg(\kr(f)) + 1)^2} \norm{\kr(f)}
        \ll_{d, r} \norm{f}.
    \]

    From the proof of \cite[Lemma~13.1.3]{friedFieldArithmetic2023},
    there exists a collection of non-zero polynomials \(c_I \in \integer[T]\),
    where \(I\) ranges over the non-empty proper subsets
    \(I \subseteq \{1, \dots, m\}\),
    with the following properties:
    \begin{itemize}
    \item \(\deg c_I \leq \deg_T f\) for all \(I\).
    \item
        If \(Q_i(t, Y)\) is irreducible in \(\rational[Y]\) for all \(i\)
        and \(c_I(t) \neq 0\) for all \(I\),
        then \(f(t, Y_1, \dots, Y_r)\)
        is irreducible in \(\rational[Y_1, \dots, Y_r]\).
    \end{itemize}
    Therefore
    \[
        S_1(f, B) \leq \sum_{1 \leq i \leq m} S_1(Q_i, B)
        + \sum_{I \subseteq \{1, \dots, m\}} \deg c_I.
    \]
    Applying Proposition~\ref{thm:hilbert_irreducibility_two_variable}
    to each \(Q_i\) gives the bound
    \[
        S_1(Q_i, B) \ll_{d, r} (\log \norm{Q_i})^5 B
        \ll_{d, r} (\log \norm{f})^5 B.
    \]
    Since \(m \leq a^{r - 1} d\) and \(\deg c_I \leq d\),
    we obtain the desired result.
\end{proof}

\subsection{Proof of Theorem~\ref{thm:hilbert_irreducibility_non_uniform}}

We proceed by induction on \(s\).
If \(s = 1\),
the theorem follows from Lemma~\ref{thm:hilbert_irreducibility_r}.
Now assume \(s > 1\)
and suppose that the theorem has been proved for \(s - 1\).

Let \(f \in \integer[T_1, \dots, T_s, Y_1, \dots, Y_r]\) be irreducible.
By applying Lemma \ref{thm:hilbert_irreducibility_r},
there are \(O_{d, r, s}((\log \norm{f})^5 B_1)\)
values of \(t \in \rational\) with \(H(t) \leq B_1\)
such that \(f(t, T_2, \dots, T_s, Y_1, \dots, Y_r)\) is reducible.
For these values,
there are trivially
\[
    O_s\left(\prod_{i = 2}^s B_i^2\right)
\]
specialisations of \((T_2, \dots, T_s)\)
that contribute to \(S(f, B_1, \dots, B_s)\).
Therefore the total contribution from this case is
\begin{equation}\label{eq:hilbert_irreducibility_reducible_case}
    O_{d, r, s}\left((\log \norm{f})^5 B_1 \prod_{i = 2}^s B_i^2\right).
\end{equation}

Now suppose that \(f(t, T_2, \dots, T_s, Y_1, \dots, Y_r)\) is irreducible.
Write \(f_t\) for the polynomial obtained by clearing denominators.
We note that \(\norm{f_t} \leq (d + 1) \norm{f} B_1^d\).
By the inductive hypothesis,
we have
\begin{align*}
    S(f_t, B_2, \dots, B_s)
    &\ll_{d, r, s} B_2 \left(\prod_{i = 3}^s B_i^2\right)
    (\log\norm{f_t} + \log B_s)^5\\
    &\ll_{d, r, s} B_2 \left(\prod_{i = 3}^s B_i^2\right)
    (\log((d + 1) \norm{f} B_1^d) + \log B_s)^5\\
    &\ll_{d, r, s} B_2 \left(\prod_{i = 3}^s B_i^2\right)
    (\log\norm{f} + \log B_s)^5.
\end{align*}
There are \(O(B_1^2)\) choices for \(t\),
so these \(t\) contribute
\[
    O_{d, r, s}\left(B_1^2 B_2 \left(\prod_{i = 3}^s B_i^2\right)
    (\log\norm{f} + \log B_s)^5\right)
\]
reducible specialisations.
Since \(B_1 \leq B_2\),
this is at most
\begin{equation}\label{eq:hilbert_irreducibility_irreducible_case}
    O_{d, r, s}\left(B_1 \left(\prod_{i = 2}^s B_i^2\right)
    (\log\norm{f} + \log B_s)^5\right).
\end{equation}
Theorem~\ref{thm:hilbert_irreducibility_non_uniform}
follows after adding the contributions from
\eqref{eq:hilbert_irreducibility_reducible_case}
and \eqref{eq:hilbert_irreducibility_irreducible_case}.

\section{Rational points on affine plane curves}
\label{sec:rational_points_curves}

The goal of this section and Section~\ref{sec:auxiliary_polynomial}
is to prove Theorem~\ref{thm:rational_points_two_variable}.
In this section,
we prove Theorem~\ref{thm:rational_points_two_variable}
assuming that the auxiliary polynomial exists.
We construct the auxiliary polynomial
in Section~\ref{sec:auxiliary_polynomial}.

\subsection{Counting function for \texorpdfstring{\(\proj^1 \times \proj^1\)}{ℙ¹ × ℙ¹}}

Instead of working with curves in \(\aff^2\) directly,
we prove a bound for curves in \(\proj^1 \times \proj^1\).
We define a counting function for curves in \(\proj^1 \times \proj^1\)
as follows.
\begin{definition}[Bihomogeneous polynomials]
    A polynomial \(f \in \integer[X_1, X_2, Y_1, Y_2]\)
    is \emph{bihomogeneous} of \emph{bidegree} \((d_1, d_2)\)
    if it is homogeneous in \((X_1, X_2)\) of degree \(d_1\)
    and homogeneous in \((Y_1, Y_2)\) of degree \(d_2\).
\end{definition}
\begin{definition}
    For a bihomogeneous polynomial \(f \in \integer[X_1, X_2, Y_1, Y_2]\),
    define
    \begin{multline*}
        \nptimesp(f, B_1, B_2) \coloneq
        \#\{((x_1 : x_2), (y_1 : y_2)) \in \proj^1 \times \proj^1 :\\
        f(x_1, x_2, y_1, y_2) = 0
        \text{ and } H(x_1 : x_2) \leq B_1
        \text{ and } H(y_1 : y_2) \leq B_2\}.
    \end{multline*}
\end{definition}

\subsection{Main results}

The main result of this section
is the following bound for the number of rational points
on an irreducible plane curve in \(\proj^1 \times \proj^1\).
\begin{theorem}\label{thm:P_times_P_bound}
    Let \(f \in \integer[X_1, X_2, Y_1, Y_2]\) be an irreducible bihomogeneous
    polynomial of bidegree \((d_1, d_2)\).
    Then
    \[
        \nptimesp(f, B_1, B_2)
        \ll d_1^7 d_2^4 B_1^{2 / d_2} (\log B_2)^2.
    \]
\end{theorem}
As a corollary,
we obtain the following bound for the number of rational points
on irreducible plane curves in \(\aff^2\).
This is the bound stated in Theorem~\ref{thm:rational_points_two_variable}.
\begin{corollary}\label{thm:rational_points_two_variable_d}
    Let \(f \in \integer[X_1, X_2]\) be an irreducible polynomial
    of degree \(d_1\) in \(X_1\) and degree \(d_2\) in \(X_2\).
    Then
    \[
        \nrat(f, B_1, B_2)
        \ll d_1^7 d_2^4 B_1^{2 / d_2} (\log B_2)^2.
    \]
\end{corollary}
In particular,
we have
\[
    \nrat(f, B_1, B_2)
    \ll_{d_1, d_2} B_1^{2 / d_2} (\log B_2)^2.
\]

To prove Theorem~\ref{thm:P_times_P_bound},
we modify the proof of \cite[Theorem~3.1]{vermeulenDimensionGrowthAffine2024}.
The approach is to use the determinant method
to construct and bound the degree of an auxiliary polynomial \(g\),
which is not divisible by \(f\)
and vanishes at the zeros of \(f\) of height at most \((B_1, B_2)\).
We construct \(g\) in Section~\ref{sec:auxiliary_polynomial}.
Finally we apply Bézout's theorem to \(f\) and \(g\)
to bound \(\nptimesp(f, B_1, B_2)\).

\subsection{Lower bounds}

Before we prove Theorem~\ref{thm:P_times_P_bound},
we give some examples of polynomials satisfying the lower bound
\(\nrat(f, B) \gg B^{2 / d_2}\).
\begin{example}
    Let \(d_2 \geq 1\).
    Consider the polynomial
    \[
        f(X_1, X_2) = X_1 X_2^{d_2} + X_1 + X_2^{d_2} \in \integer[X_1, X_2].
    \]
    Note that \(f\) is irreducible in \(\rational[X_1, X_2]\).
    Let \(a / b \in \rational \setminus \{-1\}\)
    with \(H(a / b) \leq (B / 2)^{1 / d_2}\).
    Then \(H(-a^{d_2} / (a^{d_2} + b^{d_2})) \leq B\) and
    \[
        f\left(-\frac{a^{d_2}}{a^{d_2} + b^{d_2}}, \frac{a}{b}\right) = 0,
    \]
    so \(\nrat(f, B) \gg B^{2 / {d_2}}\).
\end{example}

\begin{example}
    Consider the polynomial
    \[
        f(X_1, X_2) = X_1^2 X_2^2 + X_1^2 - X_2^2 \in \integer[X_1, X_2].
    \]
    Note that \(f\) is irreducible in \(\rational[X_1, X_2]\).
    Let \((a, b, c)\) be a primitive Pythagorean triple,
    with \(a, b, c \leq B\).
    Then \(f(a / c, a / b) = 0\).
    There are \(\Omega(B)\) such Pythagorean triples,
    see for example \cite[328]{lehmerAsymptoticEvaluationCertain1900},
    so
    \[
        \nrat(f, B) \gg B = B^{2 / \deg_2 f}.
    \]
\end{example}

\subsection{Bézout's theorem}

Bézout's theorem is the main tool used to bound the number of zeros of \(f\).
We use the following special case of Bézout's theorem,
see for example \cite[Example~4.9]{shafarevichBasicAlgebraicGeometry2013}
for a proof.
\begin{lemma}[Bézout's theorem]\label{thm:bézout}
    Let \(f \in \rational[X_1, X_2, Y_1, Y_2]\)
    be irreducible and bihomogeneous.
    Let \(g \in \rational[X_1, X_2, Y_1, Y_2]\)
    be a bihomogeneous polynomial that is not divisible by \(f\).
    Then there are at most \(\deg f \cdot \deg g\) points
    \(((x_1 : x_2), (y_1 : y_2)) \in \proj^1 \times \proj^1\)
    such that \(f(x_1, x_2, y_1, y_2) = g(x_1, x_2, y_1, y_2) = 0\).
\end{lemma}

\subsection{Reductions to special cases}

In the construction of the auxiliary polynomial,
we use a determinant estimate
that applies to absolutely irreducible polynomials \(f\),
so we need a separate argument to handle the case
where \(f\) is irreducible but not absolutely irreducible.
The following lemma shows that Theorem~\ref{thm:P_times_P_bound}
holds in this case.
\begin{lemma}\label{thm:irreducible_but_not_absolutely_irreducible_bound}
    Let \(f \in \integer[X_1, X_2, Y_1, Y_2]\) be a bihomogeneous polynomial
    of bidegree \((d_1, d_2)\).
    If \(f\) is irreducible over \(\rational\),
    but not absolutely irreducible,
    then
    \[
        \nptimesp(f, B_1, B_2) \leq (d_1 + d_2)^2.
    \]
\end{lemma}
\begin{proof}
    We use a classical argument,
    see for example \cite[Corollary~1]{heath-brownDensityRationalPoints2002}.

    First we show that if \(g \in \closure{\rational}[X_1, X_2, Y_1, Y_2]\)
    is an absolutely irreducible bihomogeneous polynomial
    of bidegree \((e_1, e_2)\)
    that is not a multiple of a rational polynomial,
    then
    \begin{equation}\label{eq:P_times_P_not_rational_multiple}
        \nptimesp(g, B_1, B_2) \leq (e_1 + e_2)^2.
    \end{equation}

    Let \(\{\lambda_1, \dots, \lambda_l\}\)
    be a basis for the \(\rational\)-vector space
    spanned by the coefficients of \(g\).
    By writing the coefficients of \(g\)
    as a rational linear combination of the \(\lambda_i\),
    we can express \(g\) in the form \(g = \sum_{i = 1}^l \lambda_i g_i\)
    where the \(g_i\) are rational bihomogeneous polynomials
    of bidegree \((e_1, e_2)\).
    Since the \(\lambda_i\) are linearly independent over \(\rational\),
    all rational zeros of \(g\) must satisfy \(g = g_i = 0\).

    Since \(g\) is absolutely irreducible
    and not a multiple of a rational polynomial,
    there is a polynomial \(g_i\) that is not a multiple of \(g\).
    By Lemma~\ref{thm:bézout},
    there are at most \((e_1 + e_2)^2\) zeros of \(g\),
    which proves \eqref{eq:P_times_P_not_rational_multiple}.

    Let \(f_1, \dots, f_k\) be the irreducible factors
    of \(f\) over \(\closure{\rational}\).
    Every zero of \(f\) must be a zero of one of the \(f_i\).
    By applying \eqref{eq:P_times_P_not_rational_multiple} to every \(f_i\),
    we get
    \begin{align*}
        \nptimesp(f, B_1, B_2)
        &\leq \sum_{i = 1}^k \nptimesp(f_i, B_1, B_2)\\
        &\leq \sum_{i = 1}^k (\deg_1 f_i + \deg_2 f_i)^2\\
        &\leq (d_1 + d_2)^2.
    \end{align*}
    This completes the proof of the lemma.
\end{proof}

Let \(c_f\) be the \(X_1^{d_1} Y_1^{d_2}\) coefficient of \(f\).
The following lemma allows us to assume that \(c_f\) is large,
which is used in the construction of the auxiliary polynomial.
It is a slight refinement of
\cite[Lemma~3.9]{vermeulenDimensionGrowthAffine2024}.
\begin{lemma}\label{thm:widetilde_f_large_coefficient}
    Let \(f \in \integer[X_1, X_2, Y_1, Y_2]\)
    be a bihomogeneous polynomial of bidegree \((d_1, d_2)\).
    Then there exists a bihomogeneous
    polynomial \(\widetilde{f} \in \integer[X_1, X_2, Y_1, Y_2]\)
    of bidegree \((d_1, d_2)\) with the following properties.
    \begin{enumerate}
    \item \(\norm{\widetilde{f}}
        \leq d_1^{d_1} d_2^{d_2} (d_1 + 1) (d_2 + 1) \norm{f}\),
    \item \(\abs{c_{\widetilde{f}}} \geq \frac{\norm{\widetilde{f}}}
        {3^{d_1 + d_2} d_1^{d_1} d_2^{d_2} (d_1 + 1) (d_2 + 1)}\),
    \item \(\nptimesp(f, B_1, B_2)
        \leq \nptimesp(\widetilde{f}, d_1 B_1, d_2 B_2)\),
    \item If \(f\) is primitive, irreducible or absolutely irreducible,
        then \(\widetilde{f}\) has the same properties.
    \end{enumerate}
\end{lemma}
\begin{proof}
    In the proof of \cite[Lemma~3.9]{vermeulenDimensionGrowthAffine2024},
    the polynomial \(\widetilde{f}\) is constructed such that
    \(\widetilde{f} = f \circ A\) for some integer matrix \(A\),
    where \(A^{-1}\) is also an integer matrix,
    and the first three properties are proved.
    This construction preserves
    primitivity, irreducibility and absolute irreducibility.
\end{proof}

\subsection{The auxiliary polynomial}

The following lemma constructs
and bounds the degree of the auxiliary polynomial.
We prove this lemma in Section~\ref{sec:auxiliary_polynomial}
using the determinant method.
\begin{lemma}\label{thm:auxiliary_polynomial}
    Let \(f \in \integer[X_1, X_2, Y_1, Y_2]\) be a primitive
    absolutely irreducible bihomogeneous polynomial of bidegree \((d_1, d_2)\).
    If \(\abs{c_f} \geq \norm{f} / C_d\),
    where \(C_d = 3^{d_1 + d_2} d_1^{d_1} d_2^{d_2} (d_1 + 1) (d_2 + 1)\),
    then there exists a bihomogeneous polynomial \(g\),
    not divisible by \(f\),
    such that \(g\) vanishes at every zero of \(f\)
    of height at most \((B_1, B_2)\),
    and
    \[
        \deg g \ll d_1^4 d_2^3 B_1^{2 / d_2} (\log B_2)^2.
    \]
\end{lemma}

\subsection{Proof of Theorem~\ref{thm:P_times_P_bound}}

By dividing by the greatest common factor of the coefficients,
we may assume that \(f\) is primitive.
If \(f\) is irreducible over \(\rational\),
but not absolutely irreducible,
then Theorem~\ref{thm:P_times_P_bound}
follows from Lemma~\ref{thm:irreducible_but_not_absolutely_irreducible_bound}.
If \(f\) is absolutely irreducible,
then applying Lemma~\ref{thm:widetilde_f_large_coefficient}
gives a primitive absolutely irreducible polynomial \(\widetilde{f}\)
of bidegree \((d_1, d_2)\) such that
\[
    \abs{c_{\widetilde{f}}} \geq \frac{\norm{\widetilde{f}}}
    {3^{d_1 + d_2} d_1^{d_1} d_2^{d_2} (d_1 + 1) (d_2 + 1)}
\]
and
\[
    \nptimesp(f, B_1, B_2)
    \leq \nptimesp(\widetilde{f}, d_1 B_1, d_2 B_2).
\]
By applying Lemma~\ref{thm:auxiliary_polynomial} to \(\widetilde{f}\),
we obtain a bihomogeneous polynomial \(g\) with
\[
    \deg g \ll d_1^4 d_2^3 (d_1 B_1)^{2 / d_2} (\log(d_2 B_2))^2,
\]
such that \(g\) is not divisible by \(\widetilde{f}\)
and \(g\) vanishes at every zero of \(\widetilde{f}\)
of height at most \((d_1 B_1, d_2 B_2)\).
Lemma~\ref{thm:bézout} implies that
\begin{align*}
    \nptimesp(f, B_1, B_2)
    &\leq \nptimesp(\widetilde{f}, d_1 B_1, d_2 B_2)\\
    &\leq \deg \widetilde{f} \cdot \deg g\\
    &\ll d_1^7 d_2^4 B_1^{2 / d_2} (\log(d_2 B_2))^2.
\end{align*}

If \(d_2 < B_2\),
then \(\log(d_2 B_2) \ll \log B_2\),
so
\[
    \nptimesp(f, B_1, B_2)
    \ll d_1^7 d_2^4 B_1^{2 / d_2} (\log B_2)^2.
\]
If \(d_2 \geq B_2\),
then for each \((y_1 : y_2) \in \proj^1\),
there are at most \(d_1\) values of \((x_1 : x_2) \in \proj^1\)
such that \(f(x_1, x_2, y_1, y_2) = 0\),
so
\begin{align*}
    \nptimesp(f, B_1, B_2)
    &\ll d_1 B_2^2\\
    &\ll d_1^7 d_2^4 B_1^{2 / d_2} (\log B_2)^2.
\end{align*}

This concludes the proof of Theorem~\ref{thm:P_times_P_bound}.

\subsection{Proof of Corollary~\ref{thm:rational_points_two_variable_d}}

Corollary~\ref{thm:rational_points_two_variable_d} follows from
Theorem~\ref{thm:P_times_P_bound} and the following result.
\begin{proposition}
    Let \(f \in \integer[X, Y]\) be a polynomial
    of degree \(d_X\) in \(X\) and degree \(d_Y\) in \(Y\).
    Let
    \[
        g(X_1, X_2, Y_1, Y_2)
        = X_2^{d_X} Y_2^{d_Y} f\left(\frac{X_1}{X_2}, \frac{Y_1}{Y_2}\right)
        \in \integer[X_1, X_2, Y_1, Y_2]
    \]
    be the homogenisation of \(f\).
    Then
    \[
        \nrat(f, B_1, B_2) \leq \nptimesp(g, B_1, B_2).
    \]
\end{proposition}
\begin{proof}
    Suppose that \((x_1 / x_2, y_1 / y_2)\)
    is a point counted by \(\nrat(f, B_1, B_2)\),
    so that \(f(x_1 / x_2, y_1 / y_2) = 0\)
    and \(H(x_1 / x_2) \leq B_1\) and \(H(y_1 / y_2) \leq B_2\).
    Then
    \[
        g(x_1, x_2, y_1, y_2)
        = x_2^{d_X} y_2^{d_Y} f\left(\frac{x_1}{x_2}, \frac{y_1}{y_2}\right)
        = 0
    \]
    and \(H(x_1 : x_2) = H(x_1 / x_2) \leq B_1\)
    and \(H(y_1 : y_2) = H(y_1 / y_2) \leq B_2\),
    so \(((x_1 : x_2), (y_1 : y_2))\)
    is a point counted by \(\nptimesp(g, B_1, B_2)\).
    Furthermore,
    these points are distinct.
\end{proof}

\section{Construction of the auxiliary polynomial}
\label{sec:auxiliary_polynomial}

In this section,
we prove Lemma~\ref{thm:auxiliary_polynomial} using the determinant method.

To prove Lemma~\ref{thm:auxiliary_polynomial},
we follow the proof of
\cite[Theorem~3.1]{vermeulenDimensionGrowthAffine2024}.
We start by assuming
that there is no auxiliary polynomial of bidegree \((M_1, M_2)\).
Then we consider determinants of matrices
consisting of the monomials of bidegree \((M_1, M_2)\)
evaluated at the zeros of \(f\) of height at most \((B_1, B_2)\).
We use an upper and lower bound on the determinant
to obtain an upper bound on \((M_1, M_2)\).

We apply the same bounds on the determinant
and construct the matrices in the same way
as in the proof of \cite[Theorem~3.1]{vermeulenDimensionGrowthAffine2024}.
However,
\citeauthor{vermeulenDimensionGrowthAffine2024}
considers polynomials satisfying \(M_2 / M_1 = d_2 / d_1\),
whereas we take \(M_2 / M_1\) to be smaller,
satisfying \eqref{eq:M_2_M_1_assumption} below.
This increases the contribution from \(B_1\)
and reduces the contribution from \(B_2\) in the upper bound.

\subsection{Determinant estimates}

Before we prove Lemma~\ref{thm:auxiliary_polynomial},
we state the bounds for the determinant that we use.

To state the bounds for the determinant,
we need to define a quantity \(b(f)\).
\begin{definition}
    For a polynomial \(f \in \integer[X_1, X_2, Y_1, Y_2]\)
    of bidegree \((d_1, d_2)\),
    we set \(b(f) \coloneq 0\) if \(f\) is not absolutely irreducible,
    and otherwise,
    we set
    \[
        b(f) \coloneq \prod_{p \in P} p^{1 / p},
    \]
    where \(P\) is the set of primes \(p > (d_1 d_2)^2\)
    such that \(f \bmod p\) is not absolutely irreducible.
\end{definition}
\begin{lemma}[{\cite[Lemma~3.6]{vermeulenDimensionGrowthAffine2024}}]
    \label{thm:b_f_upper_bound}
    Let \(f \in \integer[X_1, X_2, Y_1, Y_2]\) be primitive
    and absolutely irreducible of bidegree \((d_1, d_2)\).
    Then
    \[
        b(f) \ll \max\left(\frac{\log\norm{f}}{d_1 d_2}, 1\right)^2.
    \]
\end{lemma}
We can now state the bounds for the determinant.
First,
we state a lower bound for the determinant.
\begin{lemma}[{\cite[Proposition~3.8]{vermeulenDimensionGrowthAffine2024}}]
    \label{thm:determinant_lower_bound}
    Let \(f \in \integer[X_1, X_2, Y_1, Y_2]\) be bihomogeneous
    and absolutely irreducible.
    Let \(\xi_1, \dots, \xi_k\) be rational points on \(V(f)\).
    Let \(F_{l i} \in \integer[X_1, X_2, Y_1, Y_2]\)
    be bihomogeneous for \(1 \leq l \leq L\) and \(1 \leq i \leq k\).
    Write \(\Delta_l\) for the determinant of \((F_{l i}(\xi_j))_{i j}\).
    Let \(\Delta\) be the greatest common divisor of the \(\Delta_l\).
    If \(\Delta \neq 0\),
    then
    \[
        \log \Delta \geq \frac{k^2}{2}
        (\log k - 2 \log(d_1 d_2) - \log b(f) + O(1)).
    \]
\end{lemma}

To obtain an upper bound for the determinant,
we use the following well known
result of \textcite[Theorem~1]{bombieriSiegelsLemma1983}.
It is also known as Siegel's lemma.
\begin{lemma}\label{thm:bombieri_vaaler}
    Let \(A\) be a \(k \times m\) integer matrix of rank \(k\),
    where \(k < m\).
    Then there exists a non-zero \(x = (x_1, \dots, x_m) \in \integer^m\)
    such that \(A x = 0\) and
    \[
        \max_{1 \leq i \leq m} \abs{x_i}
        \leq (\Delta^{-1} \sqrt{\abs{\det(A A^T)}})^{1 / (m - k)},
    \]
    where \(\Delta\) is the greatest common divisor
    of all \(k \times k\) minors of \(A\).
\end{lemma}
We use the following reformulation of Lemma~\ref{thm:bombieri_vaaler}
as the upper bound for the determinants.
\begin{lemma}\label{thm:determinant_upper_bound}
    Let \(A\) be a \(k \times m\) integer matrix of rank \(k\),
    where \(k < m\).
    Let \(\Delta\) be the greatest common divisor
    of all \(k \times k\) minors of \(A\).
    Suppose that every non-zero \(x = (x_1, \dots, x_m) \in \integer^m\)
    such that \(A x = 0\) satisfies
    \(\max_{1 \leq i \leq m} \abs{x_i} \geq C\).
    Then
    \[
        \log \Delta \leq \frac{1}{2} \log\abs{\det(A A^T)} - (m - k) \log C.
    \]
\end{lemma}
\begin{proof}
    By Lemma~\ref{thm:bombieri_vaaler},
    there exists a non-zero \(x = (x_1, \dots, x_m) \in \integer^m\)
    such that \(A x = 0\) and
    \[
        \max_{1 \leq i \leq m} \abs{x_i}
        \leq (\Delta^{-1} \sqrt{\abs{\det(A A^T)}})^{1 / (m - k)}.
    \]
    Applying the assumption \(\max_{1 \leq i \leq m} \abs{x_i} \geq C\),
    we obtain
    \[
        C \leq (\Delta^{-1} \sqrt{\abs{\det(A A^T)}})^{1 / (m - k)}.
    \]
    Rearranging gives
    \[
        \Delta \leq \sqrt{\abs{\det(A A^T)}} C^{-(m - k)}.
    \]
    The lemma follows after taking the logarithm of both sides.
\end{proof}

\subsection{The setup}

The rest of this section is devoted
to the proof of Lemma~\ref{thm:auxiliary_polynomial}.

Let \(S \subseteq \proj^1 \times \proj^1\)
be the set of rational points of height at most \((B_1, B_2)\)
on the curve defined by \(f\).
Let \(M_1\) and \(M_2\) be such that there is
no bihomogeneous polynomial \(g\) of bidegree \((M_1, M_2)\),
where \(g\) is not divisible by \(f\)
and \(g\) vanishes at every point of \(S\).
Our goal is to deduce an upper bound on \((M_1, M_2)\).

The first step is to construct a matrix \(A\)
consisting of monomials of bidegree \((M_1, M_2)\)
evaluated at the points of \(S\).

The next step is to apply the determinant estimates
of Lemma~\ref{thm:determinant_lower_bound}
and Lemma~\ref{thm:determinant_upper_bound}
and manipulate the estimates to give an upper bound for \(M_1\).
In order to manipulate the estimates to give the desired upper bound,
we need to make some assumptions on the value of \(M_1\) and \(M_2\).
We assume that
\begin{equation}\label{eq:M_1_assumption}
    M_1 \geq K_0 (d_1^2 d_2^2 + d_1 d_2 \log B_1 \log B_2)
\end{equation}
for some \(K_0\) to be specified later,
and that
\begin{equation}\label{eq:M_2_M_1_assumption}
    \frac{1}{d_1 \log B_2} \leq \frac{M_2}{M_1}
    \leq \min\left(\frac{1}{d_1}, \frac{1}{\log B_2}\right).
\end{equation}

The final step is to construct the auxiliary polynomial \(g\)
such that the degree satisfies
the assumptions \eqref{eq:M_1_assumption} and \eqref{eq:M_2_M_1_assumption},
but not the upper bound on \(M_1\) and \(M_2\).
This guarantees that \(g\) is not divisible by \(f\)
and \(g\) vanishes at every point of \(S\).

\subsection{Construction of the matrix}

First,
we construct the matrix \(A\) and compute its dimensions.

For \(D_1, D_2 \geq 0\),
write \(\monomials[D_1, D_2]\) for the set of bihomogeneous monomials
of bidegree \((D_1, D_2)\),
so \(\abs{\monomials[D_1, D_2]} = (D_1 + 1) (D_2 + 1)\).
For each \(x \in S\),
we can construct a vector \(v_x\)
by applying all monomials in \(\monomials[M_1, M_2]\) to \(x\).
Let \(S' \subseteq S\) be a maximal subset
such that the set \(\{v_x : x \in S'\}\) is linearly independent.
Let \(A\) be the \(k \times m\) matrix
where the rows are given by \(v_x\) for \(x \in S'\).

The number of columns in \(A\) is the number of monomials
in \(\monomials[M_1, M_2]\),
so we have
\begin{equation}\label{eq:m_value}
    m = \abs{\monomials[M_1, M_2]} = (M_1 + 1)(M_2 + 1).
\end{equation}

The solutions to \(A x = 0\)
correspond to the coefficients of the bihomogeneous polynomials
of bidegree \((M_1, M_2)\) that vanish on every point of \(S'\)
and therefore vanish on every point of \(S\).
By our assumption on \((M_1, M_2)\),
these polynomials are exactly the multiples of \(f\)
of bidegree \((M_1, M_2)\).
The assumptions \eqref{eq:M_1_assumption} and \eqref{eq:M_2_M_1_assumption}
imply that \(M_1 \geq d_1\) and \(M_2 \geq d_2\),
so a basis for this set is
\[
    \{f p : p \in \monomials[M_1 - d_1, M_2 - d_2]\}.
\]
Therefore
\begin{equation}\label{eq:k_value}
\begin{split}
    k &= \abs{\monomials[M_1, M_2]} - \abs{\monomials[M_1 - d_1, M_2 - d_2]}\\
    &= d_2 M_1 + d_1 M_2 - d_1 d_2 + d_1 + d_2.
\end{split}
\end{equation}

\subsection{Applying the determinant lower bound}

Next,
we apply the lower bound for the determinant of \(A\)
and manipulate the bound to give an upper bound for \(\log M_1\).

Let \(\Delta\) be the greatest common divisor
of all \(k \times k\) minors of \(A\).
The rows of \(A\) are linearly independent,
so \(\Delta \neq 0\).
Applying Lemma~\ref{thm:determinant_lower_bound}
where the \(\Delta_l\) ranges over the \(k \times k\) minors of \(A\),
we obtain
\[
    \log \Delta \geq \frac{k^2}{2}
    (\log k - 2 \log(d_1 d_2) - \log b(f) + C_0)
\]
for some constant \(C_0\).
Dividing by \(M_1 k\) gives
\begin{equation}\label{eq:log_delta_initial_lower_bound}
    \frac{\log \Delta}{M_1 k} \geq \frac{k}{2 M_1}
    (\log k - 2 \log(d_1 d_2) - \log b(f) + C_0).
\end{equation}

From \eqref{eq:k_value} and the fact that \(M_1 \geq d_1\),
which comes from the assumption \eqref{eq:M_1_assumption},
we have that \(k - d_1 M_2 > 0\).

If the right hand side of \eqref{eq:log_delta_initial_lower_bound}
is non-negative,
then we can replace the \(k\) with \(k - d_1 M_2\)
in the numerator of the right hand side,
because it makes the right hand side smaller.

On the other hand,
if the right hand side of \eqref{eq:log_delta_initial_lower_bound} is negative,
we can replace the \(k\) with \(k - d_1 M_2\)
because the left hand side is non-negative,
while the right hand side remains negative
after replacing \(k\) with \(k - d_1 M_2 > 0\).

Putting this together yields
\begin{equation}\label{eq:log_delta_final_lower_bound}
    \frac{\log \Delta}{M_1 k} \geq \frac{k - d_1 M_2}{2 M_1}
    (\log k - 2 \log(d_1 d_2) - \log b(f) + C_0).
\end{equation}

Now we focus on the right hand side of \eqref{eq:log_delta_final_lower_bound}.
By \eqref{eq:k_value},
we have
\[
    \frac{k - d_1 M_2}{M_1} = d_2 + O\left(\frac{d_1 d_2}{M_1}\right).
\]
By the assumptions \eqref{eq:M_1_assumption} and \eqref{eq:M_2_M_1_assumption},
we have \(d_1 M_2 / M_1 \leq 1\) and \(d_1 d_2 / M_1 \leq 1\).
Together with \eqref{eq:k_value},
this implies that
\begin{equation}\label{eq:log_k_value}
    \log k = \log M_1 + \log d_2 + O(1).
\end{equation}
Thus we can replace the right hand side of
\eqref{eq:log_delta_final_lower_bound} with
\begin{multline*}
    \frac{d_2}{2} (\log M_1 - 2 \log d_1 - \log d_2 - \log b(f) + O(1))\\
    + O\left(\frac{d_1 d_2 (\log M_1 + \log d_1 + \log d_2)}{M_1}
    + \frac{d_1 d_2 \log b(f)}{M_1}\right).
\end{multline*}
By the assumption \eqref{eq:M_1_assumption},
the first error term is \(O(1)\).
Putting this into \eqref{eq:log_delta_final_lower_bound} and rearranging,
we obtain
\begin{equation}\label{eq:log_M_1_initial_upper_bound}
    \log M_1 \leq \frac{2 \log \Delta}{d_2 M_1 k} + 2 \log d_1 + \log d_2
    + \left(1 + O\left(\frac{d_1}{M_1}\right)\right) \log b(f) + O(1).
\end{equation}

\subsection{Applying the determinant upper bound}

Next,
we apply the upper bound for the determinant of \(A\).

We know that the solutions to \(A x = 0\)
correspond to the coefficients of multiples of \(f\)
of bidegree \((M_1, M_2)\).
Since \(f\) is primitive,
every multiple of \(f\) with integer coefficients must be of the form
\(f p\) for some \(p \in \integer[X_1, X_2, Y_1, Y_2]\).
The coefficient of \(f p\) with the greatest degree in \(X_1\) and \(Y_1\)
must then have size at least \(\abs{c_f} \geq \norm{f} / C_d\).

Therefore every integer solution to \(A x = 0\)
satisfies \(\max_{1 \leq i \leq m} \abs{x_i} \geq \norm{f} / C_d\).
Since the rows of \(A\) are linearly independent,
\(A\) has rank \(k\),
so we can apply Lemma~\ref{thm:determinant_upper_bound}
to obtain the bound
\[
    \log \Delta \leq \frac{1}{2} \log\abs{\det(A A^T)}
    - (m - k) (\log\norm{f} - \log C_d).
\]
The entries of \(A\) are bounded by \(B_1^{M_1} B_2^{M_2}\),
which gives the estimate
\[
    \abs{\det(A A^T)} \leq k! (m B_1^{2 M_1} B_2^{2 M_2})^k.
\]
Putting this together,
we get
\begin{multline*}
    \log \Delta
    \leq \frac{\log k!}{2} + \frac{k}{2} \log m + k M_1 \log B_1
    + k M_2 \log B_2\\
    - (m - k) (\log\norm{f} - \log C_d).
\end{multline*}
Now we multiply both sides by \(2 / (d_2 M_1 k)\) so that
the left hand side matches the term in \eqref{eq:log_M_1_initial_upper_bound}.
This gives
\begin{multline*}
    \frac{2 \log \Delta}{d_2 M_1 k}
    \leq \frac{\log k!}{d_2 M_1 k} + \frac{\log m}{d_2 M_1}
    + \frac{2}{d_2} \log B_1 + \frac{2 M_2}{d_2 M_1} \log B_2\\
    - \frac{2 (m - k)}{d_2 M_1 k} (\log\norm{f} - \log C_d).
\end{multline*}

From the estimate \(\log k! \leq k \log k\) and \eqref{eq:log_k_value},
we see that the first term on the right hand side is \(O(1)\).
The assumption \eqref{eq:M_2_M_1_assumption} implies that \(M_2 \leq M_1\),
so that \(\log m \ll \log M_1 + \log M_2 \ll \log M_1\),
so the second term is also \(O(1)\).
Additionally,
\eqref{eq:M_2_M_1_assumption} implies that \(M_2 / M_1 \leq 1 / \log B_2\),
so the fourth term is also \(O(1)\).
Hence we have
\begin{equation}\label{eq:log_delta_upper_bound}
    \frac{2 \log \Delta}{d_2 M_1 k} \leq \frac{2}{d_2} \log B_1
    - \frac{2 (m - k)}{d_2 M_1 k} (\log\norm{f} - \log C_d) + O(1).
\end{equation}

Next,
we estimate \((m - k) / (M_1 k)\).
By \eqref{eq:m_value} and \eqref{eq:k_value},
we have
\[
    m - k = M_1 M_2 + O(d_2 M_1 + d_1 M_2)
\]
and
\[
    M_1 k = d_2 M_1^2 + d_1 M_1 M_2 + O(d_1 d_2 M_1).
\]
This gives
\begin{align*}
    \frac{m - k}{M_1 k}
    &= \frac{M_1 M_2 + O(d_2 M_1 + d_1 M_2)}{d_2 M_1^2 + d_1 M_1 M_2
    + O(d_1 d_2 M_1)}\\
    &= \frac{M_1 M_2}{d_2 M_1^2 + d_1 M_1 M_2 + O(d_1 d_2 M_1)}
    + O\left(\frac{1}{M_1}\right)\\
    &= \frac{M_2}{d_2 M_1 + d_1 M_2} + O\left(\frac{1}{M_1}
    + \frac{d_1 d_2 M_2}{(d_2 M_1 + d_1 M_2)^2}\right)\\
    &= \frac{M_2}{d_2 M_1 + d_1 M_2} + O\left(\frac{1}{M_1}\right).
\end{align*}
By the assumption \eqref{eq:M_2_M_1_assumption},
we have \(1 / (d_1 \log B_2) \leq M_2 / M_1 \leq 1 / d_1\),
so
\[
    \frac{1}{d_1 (d_2 \log B_2 + 1)}
    \leq \frac{M_2}{d_2 M_1 + d_1 M_2}
    \leq \frac{1}{d_1 d_2}.
\]
Therefore
\[
    -\frac{2 (m - k)}{d_2 M_1 k} \log\norm{f}
    \leq -\frac{2 \log\norm{f}}{d_1 d_2 (d_2 \log B_2 + 1)}
    + O\left(\frac{\log\norm{f}}{d_2 M_1}\right).
\]
Recall that \(C_d = 3^{d_1 + d_2} d_1^{d_1} d_2^{d_2} (d_1 + 1) (d_2 + 1)\).
Using the assumption \eqref{eq:M_1_assumption},
we have
\begin{align*}
    \frac{2 (m - k)}{d_2 M_1 k} \log C_d
    &\leq \left(\frac{2}{d_1 d_2^2} + O\left(\frac{1}{d_2 M_1}\right)\right)
    (d_1 \log d_1 + d_2 \log d_2 + O(d_1 + d_2))\\
    &\leq 2 \log d_1 + \frac{2}{d_2} \log d_2 + O(1)\\
    &\leq 2 \log d_1 + O(1).
\end{align*}

Putting this into \eqref{eq:log_delta_upper_bound},
we obtain
\begin{multline*}
    \frac{2 \log \Delta}{d_2 M_1 k} \leq \frac{2}{d_2} \log B_1 + 2 \log d_1
    -\frac{2 \log\norm{f}}{d_1 d_2 (d_2 \log B_2 + 1)}\\
    + O\left(\frac{\log\norm{f}}{d_2 M_1}\right) + O(1).
\end{multline*}
We now put this into \eqref{eq:log_M_1_initial_upper_bound},
which gives
\begin{multline}\label{eq:log_M_1_final_upper_bound}
    \log M_1 \leq \frac{2}{d_2} \log B_1 + 4 \log d_1 + \log d_2
    -\frac{2 \log\norm{f}}{d_1 d_2 (d_2 \log B_2 + 1)}\\
    + \left(1 + O\left(\frac{d_1}{M_1}\right)\right) \log b(f)
    + O\left(\frac{\log\norm{f}}{d_2 M_1}\right) + O(1).
\end{multline}

\subsection{An upper bound for \texorpdfstring{\(M_1\)}{M₁}}

We now use \eqref{eq:log_M_1_final_upper_bound}
to obtain an upper bound on \(M_1\).

We take the \(K_0\) in \eqref{eq:M_1_assumption} sufficiently large,
so that the terms in \eqref{eq:log_M_1_final_upper_bound} satisfy
\[
    O\left(\frac{d_1}{M_1}\right) \leq \frac{1}{d_2 \log B_2}
\]
and
\[
    O\left(\frac{\log\norm{f}}{d_2 M_1}\right)
    \leq \frac{\log\norm{f}}{d_1 d_2 (d_2 \log B_2 + 1)}.
\]
Then \eqref{eq:log_M_1_final_upper_bound} becomes
\begin{multline*}
    \log M_1 \leq \frac{2}{d_2} \log B_1 + 4 \log d_1 + \log d_2
    -\frac{\log\norm{f}}{d_1 d_2 (d_2 \log B_2 + 1)}\\
    + \left(1 + \frac{1}{d_2 \log B_2}\right) \log b(f) + O(1).
\end{multline*}

By \cite[Lemma~3.5.5]{castryckDimensionGrowthConjecture2020},
for any \(x > 1\) and \(a, c > 0\),
we have
\[
    a \log \log x - \frac{\log x}{c}
    = a \left(\log \log x - \frac{\log x}{a c}\right)
    \leq a (\log c + \log a + O(1)).
\]
Taking \(a = 2 (1 + 1 / (d_2 \log B_2))\)
and \(c = d_1 d_2 (d_2 \log B_2 + 1)\)
and applying Lemma~\ref{thm:b_f_upper_bound} gives
\begin{multline*}
    \left(1 + \frac{1}{d_2 \log B_2}\right) \log b(f)
    - \frac{\log\norm{f}}{d_1 d_2 (d_2 \log B_2 + 1)}\\
    \begin{aligned}
        &\leq a \max\{\log\log\norm{f} - \log d_1 - \log d_2, 0\}
        - \frac{\log\norm{f}}{c} + O(1)\\
        &\leq a \max\{\log c - \log d_1 - \log d_2 + \log a + O(1), 0\}
        + O(1)\\
        &\leq 2 \left(1 + \frac{1}{d_2 \log B_2}\right)
        (\log \log B_2 + \log d_2) + O(1)\\
        &\leq 2 \log \log B_2 + 2 \log d_2 + O(1).
    \end{aligned}
\end{multline*}
Hence
\[
    \log M_1 \leq \frac{2}{d_2} \log B_1 + 2 \log \log B_2
    + 4 \log d_1 + 3 \log d_2 + O(1),
\]
which implies that
\begin{equation}\label{eq:M_1_upper_bound}
    M_1 \ll d_1^4 d_2^3 B_1^{2 / d_2} (\log B_2)^2.
\end{equation}

\subsection{Construction of \texorpdfstring{\(g\)}{g}}

Finally,
we construct the auxiliary polynomial \(g\).
Consider the bihomogeneous polynomials of bidegree \((M_1, M_2)\)
where
\[
    M_1 = K d_1^4 d_2^3 B_1^{2 / d_2} (\log B_2)^2
\]
for some large constant \(K\),
and
\[
    \frac{1}{d_1 \log B_2} \leq \frac{M_2}{M_1}
    \leq \min\left(\frac{1}{d_1}, \frac{1}{\log B_2}\right).
\]
Notice that if \(K\) is sufficiently large,
\((M_1, M_2)\) satisfies the assumptions \eqref{eq:M_1_assumption}
and \eqref{eq:M_2_M_1_assumption}
because if \(d_2 \geq \log B_1 / \log \log B_1\),
then \(d_2^2 \geq \log B_1\),
and otherwise \(B_1^{2 / d_2} \geq \log B_1\).
However,
if \(K\) is larger
than the implied constant of \eqref{eq:M_1_upper_bound},
\((M_1, M_2)\) does not satisfy \eqref{eq:M_1_upper_bound}.

This proves the existence of a bihomogeneous polynomial \(g\)
of bidegree \((M_1, M_2)\) that vanishes at every zero of \(f\) of height
at most \((B_1, B_2)\),
which completes the proof of Lemma~\ref{thm:auxiliary_polynomial}.

\section{Rational points on irreducible affine hypersurfaces}
\label{sec:rational_points_hypersurface}

The goal of this section is to prove
Theorems~\ref{thm:rational_points_hypersurface}
and \ref{thm:rational_points_hypersurface_monic}.

\subsection{Proof outline}

In the proof of Theorem~\ref{thm:rational_points_hypersurface},
we consider specialisations of \(X_3, \dots, X_n\) and bound \(\nrat(f, B)\)
by bounding the number of points coming from each specialisation.
Then we split into the cases where the specialisation is irreducible,
reducible and zero.

When the specialisation is irreducible,
we use Theorem~\ref{thm:rational_points_two_variable}
to bound the number of points.
We trivially bound the number of these specialisations
by \(O_n(B^{2 (n - 2)})\).

When the specialisation is reducible,
we bound the number of points using the Schwartz--Zippel bound,
see Lemma~\ref{thm:schwartz_zippel_rational}.
The number of these specialisations
is bounded using Corollary~\ref{thm:hilbert_irreducibility}.

When the specialisation is the zero polynomial,
we trivially bound the number of points by \(O(B^4)\).
To bound the number of these specialisations,
we consider the specialisations of \(X_n\) and use induction.

We use a similar approach
to prove Theorem~\ref{thm:rational_points_hypersurface_monic},
but we take more care when bounding the contribution
from the reducible case.

\subsection{Schwartz--Zippel bound}

To bound the number of points coming from the reducible case,
we use the following bound,
which is known as the Schwartz--Zippel bound or the trivial bound.
We state the bound in the form of \cite[Lemma~14]{bukhSumProductEstimates2012},
which gives a bound on \(\nrat(f, B)\) as a special case.

\begin{lemma}[Schwartz--Zippel bound]\label{thm:schwartz_zippel}
    Let \(X \subseteq \aff^n\) be a variety of dimension \(m\)
    and degree \(d\).
    Let \(A_1, \dots, A_n \subseteq \aff^1\) be finite sets of the same size.
    Then
    \[
        \#\{(x_1, \dots, x_n) \in X : x_i \in A_i
        \text{ for } i = 1, \dots, n\}
        \leq d \abs{A_1}^m.
    \]
\end{lemma}
\begin{lemma}\label{thm:schwartz_zippel_rational}
    Let \(f \in \integer[X_1, \dots, X_n]\)
    be a non-zero polynomial of degree \(d\).
    Then
    \[
        \nrat(f, B) \ll_{n} d B^{2 (n - 1)}.
    \]
\end{lemma}
\begin{proof}
    Apply Lemma~\ref{thm:schwartz_zippel}
    with \(A_1 = \dots = A_n = \{t \in \rational : H(t) \leq B\}\).
\end{proof}

\subsection{Proof of Theorem~\ref{thm:rational_points_hypersurface}}

First,
we bound the number of specialisations
that result in the zero polynomial.
\begin{lemma}\label{thm:f_specialisation_non_zero}
    Let \(n \geq 3\).
    Let \(f \in \integer[X_1, \dots, X_n]\)
    be an irreducible polynomial of degree \(d\).
    Suppose that \(f\) has degree at least one in \(X_1\).
    The number of \((x_3, \dots, x_n) \in \rational^{n - 2}\)
    such that \(H(x_i) \leq B\) for all \(i\)
    and \(f(X_1, X_2, x_3, \dots, x_n)\) is the zero polynomial
    is at most
    \[
        O_{d, n}(B^{2 n - 7} (\log\norm{f} + \log B)^5).
    \]
\end{lemma}
\begin{proof}
    For \(x \in \rational\),
    let \(f_x\) be the polynomial given by clearing the denominators
    of \(f(X_1, \dots, X_{n - 1}, x)\).
    Then \(f_x\) is not the zero polynomial,
    because if \(f_x = 0\),
    then \(f\) is divisible by \(X_n - x\).
    Since \(f\) is irreducible over \(\rational\),
    this implies that \(f\) is a constant multiple of \(X_n - x\),
    which is not possible because \(f\) has degree at least one in \(X_1\).

    Thus for \(n = 3\),
    we are done.
    Now assume that \(n > 3\)
    and the lemma has been proved for \(n - 1\).

    \textbf{Case 1:}
    Consider the case where \(f_x\) is irreducible and \(\deg_1 f_x \geq 1\).
    When \(H(x) \leq B\),
    we have \(\norm{f_x} \leq (d + 1) \norm{f} B^d\),
    so \(\log\norm{f_x} \ll_d \log\norm{f} + \log B\).
    By induction,
    there are at most
    \[
        O_{d, n}(B^{2 n - 9} (\log\norm{f_x} + \log B)^5)
        = O_{d, n}(B^{2 n - 9} (\log\norm{f} + \log B)^5)
    \]
    values of \((x_3, \dots, x_{n - 1})\)
    such that \(f_x(X_1, X_2, x_3, \dots, x_{n - 1}) = 0\)
    and \(H(x_i) \leq B\) for all \(i\).
    There are \(O(B^2)\) choices for \(x\),
    so the total contribution from this case is
    \[
        O_{d, n}(B^{2 n - 7 } (\log\norm{f} + \log B)^5).
    \]

    \textbf{Case 2:}
    Next consider the case where \(f_x\) is reducible or \(\deg_1 f_x = 0\).
    If we treat \(f_x\) as a polynomial
    with coefficients in \(\integer[X_3, \dots, X_{n - 1}]\)
    and variables \(X_1\) and \(X_2\),
    there must be a non-zero coefficient \(g\).
    In order for \(f_x(X_1, X_2, x_3, \dots, x_{n - 1}) = 0\),
    we must have \(g(x_3, \dots, x_{n - 1}) = 0\).
    By Lemma~\ref{thm:schwartz_zippel_rational},
    this occurs for at most \(O_{d, n}(B^{2 (n - 4)})\)
    values of \((x_3, \dots, x_{n - 1})\)
    such that \(H(x_i) \leq B\) for all \(i\).

    \(f_x\) is reducible for at most \(O_{d, n}((\log\norm{f})^5 B)\)
    values of \(x\) with \(H(x) \leq B\)
    by Lemma~\ref{thm:hilbert_irreducibility_r}.
    Also,
    \(\deg_1 f_x = 0\) occurs for at most \(d\) values of \(x\)
    because if we consider \(f\) as a polynomial with coefficients
    in \(\integer[X_n]\),
    this requires \(x\) to be a root of the coefficients
    for the terms of degree at least one in \(X_1\).

    Therefore the total contribution from the reducible case is
    \[
        O_{d, n}((\log\norm{f})^5 B^{2 n - 7}).
    \]
    Lemma~\ref{thm:f_specialisation_non_zero}
    follows after adding the contribution from the two cases.
\end{proof}

We now prove Theorem~\ref{thm:rational_points_hypersurface}.
Let \(d_1 \geq 2\) be the degree of \(f\) in \(X_1\).
We consider specialisations of \(f\) of the form
\[
    g(X_1, X_2) \coloneq f(X_1, X_2, x_3, \dots, x_n),
\]
where \((x_3, \dots, x_n) \in \rational^{n - 2}\).

\textbf{Case 1:}
If \(g\) is irreducible and \(\deg_1 g = d_1\),
then we have
\[
    \nrat(g, B) \ll_d B^{2 / d_1} (\log B)^2
    \ll B (\log B)^2
\]
by Theorem~\ref{thm:rational_points_two_variable}.
There are \(O_n(B^{2 (n - 2)})\) values of \((x_3, \dots, x_n)\)
with \(H(x_i) \leq B\) for all \(i\),
so the total contribution from this case is
\[
    O_{d, n}(B^{2 n - 3} (\log B)^2).
\]

\textbf{Case 2:}
Next,
consider the case where \(g\) is the zero polynomial.
There are \(O(B^4)\) specialisations of \(X_1\) and \(X_2\)
with \(H(x_i) \leq B\) for \(i = 1, 2\).
By Lemma~\ref{thm:f_specialisation_non_zero},
the total contribution from this case is
\[
    O_{d, n}(B^{2 n - 3} (\log\norm{f} + \log B)^5).
\]

\textbf{Case 3:}
Finally,
consider the case where \(g\) is non-zero
and \(g\) is either reducible or \(\deg_1 g < d_1\).
By Lemma~\ref{thm:schwartz_zippel_rational},
we have
\[
    \nrat(g, B) \ll_d B^2.
\]
By Corollary~\ref{thm:hilbert_irreducibility},
there are
\[
    O_{d, n}(B^{2 n - 5} (\log\norm{f} + \log B)^5)
\]
values of \((x_3, \dots, x_n)\) such that \(g\) is reducible
and \(H(x_i) \leq B\) for all \(i\).
In order for \(\deg_1 g < d_1\),
when we consider \(f\) as a polynomial
with coefficients in \(\integer[X_3, \dots, X_n]\),
the coefficient of a leading term in \(X_1\) must vanish.
By Lemma~\ref{thm:schwartz_zippel_rational},
this occurs at most \(O_{d, n}(B^{2 (n - 3)})\) times.
Therefore,
the total contribution from this case is
\[
    O_{d, n}(B^{2 n - 3} (\log\norm{f} + \log B)^5).
\]

Theorem~\ref{thm:rational_points_hypersurface}
follows after adding the contribution from all the cases.

\subsection{Proof of Theorem~\ref{thm:rational_points_hypersurface_monic}}

In the proof of Theorem~\ref{thm:rational_points_hypersurface_monic},
we use the following lemma.
The proof of the lemma follows a similar approach to the proof of
Theorem~\ref{thm:rational_points_hypersurface},
but we take more care when bounding the points coming from the reducible case.
Specifically,
we split into the case where there is a factor that is linear in \(Y\)
and the case where every factor has degree at least two in \(Y\).
Additionally,
the case where the specialisation is the zero polynomial does not occur
because the leading term in \(Y\) does not vanish.
\begin{lemma}\label{thm:number_rational_specialisation_with_root_multivariable}
    Let \(f \in \integer[Y, T_1, \dots, T_n]\) be an irreducible polynomial
    of degree \(d\) and degree \(d_Y\) in \(Y\).
    Suppose that the leading coefficient in \(Y\) is constant.
    Write \(R(f, B)\) for the number of \((t_1, \dots, t_n) \in \rational^n\)
    with \(H(t_i) \leq B\)
    such that \(f(Y, t_1, \dots, t_n)\) has a root in \(\rational\).
    Then
    \[
        R(f, B) \ll_{d, n} B^{2 (n - 1 + 1 / d_Y)} (\log\norm{f} + \log B)^7.
    \]
\end{lemma}
\begin{proof}
    We adapt the proof of
    \cite[Lemma~7]{castilloHilbertsIrreducibilityTheorem2017}.
    The proof is by induction.

    The base case is \(n = 1\).
    Suppose that \(t \in \rational\) and \(f(Y, t)\) has a rational root \(y\).
    Let \(f_t\) be the polynomial obtained by clearing denominators
    in \(f(Y, t)\).
    Then \(\norm{f_t} \leq (d + 1) \norm{f} H(t)^d\).
    Since \(f\) has constant leading coefficient in \(Y\),
    we see that \(f_t\) is non-zero,
    so by the rational root theorem,
    we have \(H(y) \leq \norm{f_t} \leq (d + 1) \norm{f} H(t)^d\).
    So the number of \(t\) contributing to \(R(f, B)\) is bounded by
    \(\nrat(f, (d + 1) \norm{f} B^d, B)\).
    By Theorem~\ref{thm:rational_points_two_variable},
    we have
    \[
        R(f, B)
        \leq \nrat(f, (d + 1) \norm{f} B^d, B)
        \ll_d B^{2 / d_Y} (\log\norm{f} + \log B)^2.
    \]

    Now assume that the lemma has been proved for \(n - 1\).
    For \((t_2, \dots, t_n) \in \rational^{n - 1}\),
    let \(g \in \integer[Y, T_1]\) be the polynomial obtained
    by clearing denominators in \(f(Y, T_1, t_2, \dots, t_n)\).
    Since \(f\) has constant leading coefficient in \(Y\),
    we see that \(g\) also has constant leading coefficient in \(Y\)
    and \(\deg_Y g = d_Y\).
    If \(H(t_i) \leq B\) for all \(i\),
    then we have
    \begin{equation}\label{eq:log_norm_g}
        \log\norm{g} \ll_{d, n} \log\norm{f} + \log B.
    \end{equation}

    \textbf{Case 1:}
    First,
    consider the case where \(g\) is irreducible.
    Applying the \(n = 1\) case of the lemma to \(g\)
    and using \eqref{eq:log_norm_g} gives
    \begin{align*}
        R(g, B)
        &\ll_{d, n} B^{2 / \deg_Y g} (\log\norm{g} + \log B)^2\\
        &\ll_{d, n} B^{2 / d_Y} (\log\norm{f} + \log B)^2.
    \end{align*}
    There are \(O(B^{2 (n - 1)})\) values of \((t_2, \dots, t_n)\)
    such that \(H(t_i) \leq B\) for all \(i\),
    so the total contribution to \(R(f, B)\) from this case is
    \[
        O_{d, n}(B^{2 (n - 1 + 1 / d_Y)} (\log\norm{f} + \log B)^2).
    \]

    \textbf{Case 2:}
    Next,
    consider the case where \(g\) is reducible over \(\rational\).
    Since \(f\) has constant leading coefficient in \(Y\),
    every factor of \(g\) must have degree at least one in \(Y\)
    and the factors also have constant leading coefficient in \(Y\).
    Suppose that every irreducible factor of \(g\)
    has degree at least two in \(Y\).
    By
    \cite[Chapter~3, Proposition~2.3]{langFundamentalsDiophantineGeometry1983}
    and \eqref{eq:log_norm_g},
    every factor \(h\) of \(g\) satisfies
    \(\log\norm{h} \ll_{d, n} \log\norm{g} \ll_{d, n} \log\norm{f} + \log B\).
    Then applying the \(n = 1\) case of the lemma
    to each irreducible factor of \(g\) gives
    \[
        R(g, B) \ll_{d, n} B (\log\norm{f} + \log B)^2.
    \]
    By Corollary~\ref{thm:hilbert_irreducibility},
    there are
    \[
        O_{d, n}(B^{2 n - 3} (\log\norm{f} + \log B)^5)
    \]
    values of \((t_2, \dots, t_n)\) such that \(g\) is reducible
    and \(H(t_i) \leq B\) for all \(i\),
    so the total contribution to \(R(f, B)\) from these values is
    \[
        O_{d, n}(B^{2 (n - 1)} (\log\norm{f} + \log B)^7).
    \]

    \textbf{Case 3:}
    Finally,
    consider the case where \(g\) is reducible over \(\rational\)
    and there is a factor that is linear in \(Y\).
    Let \(U\) be the set of \((t_2, \dots, t_n) \in \rational^{n - 1}\)
    such that \(H(t_i) \leq B\) for all \(i\) and this case occurs.

    By \cite[Corollary~1.11]{cluckersImprovementsDimensionGrowth2025},
    there exists \(t_1 \in \integer\) such that
    \[
        f_{t_1}(Y, T_2, \dots, T_n) \coloneq f(Y, t_1, T_2, \dots, T_n)
    \]
    is irreducible and \(\abs{t_1}\) is at most polynomial in \(\log\norm{f}\),
    where the degree and coefficients depend on \(d\) and \(n\).
    Then \(\norm{f_{t_1}} \leq (d + 1) \norm{f} \abs{t_1}^d\),
    which implies that \(\log\norm{f_{t_1}} \ll_{d, n} \log\norm{f}\).
    Also, \(\deg_Y f_{t_1} = d_Y\)
    because \(f\) has constant leading coefficient in \(Y\).

    If \((t_2, \dots, t_n) \in U\),
    then
    \[
        f_{t_1}(Y, t_2, \dots, t_n) = f(Y, t_1, t_2, \dots, t_n) = g(Y, t_1)
    \]
    must have a rational root
    because \(g(Y, t_1)\) has a linear factor over \(\rational\).
    Therefore every \((t_2, \dots, t_n) \in U\)
    is counted by \(R(f_{t_1}, B)\).
    By induction,
    we have
    \begin{align*}
        \abs{U}
        &\leq R(f_{t_1}, B)\\
        &\ll_{d, n} B^{2 (n - 2 + 1 / \deg_Y f_{t_1})}
        (\log\norm{f_{t_1}} + \log B)^7\\
        &\ll_{d, n} B^{2 (n - 2 + 1 / d_Y)} (\log\norm{f} + \log B)^7.
    \end{align*}

    For each \((t_2, \dots, t_n) \in U\),
    we trivially bound the number of \(t_1\) contributing to \(R(f, B)\)
    by \(O(B^2)\).
    Thus the total contribution from this case is
    \[
        O_{d, n}(B^{2 (n - 1 + 1 / d_Y)} (\log\norm{f} + \log B)^7).
    \]

    This completes the proof of the lemma.
\end{proof}

We now prove Theorem~\ref{thm:rational_points_hypersurface_monic}.

Suppose that \((x_1, \dots, x_n)\) contributes to \(\nrat(f, B)\).
Then \(x_1\) is a rational root of the polynomial \(f(X_1, x_2, \dots, x_n)\).
By Lemma~\ref{thm:number_rational_specialisation_with_root_multivariable},
\(f(X_1, x_2, \dots, x_n)\) has a rational root for at most
\[
    O_{d, n}(B^{2 (n - 2 + 1 / d_1)} (\log\norm{f} + \log B)^7)
\]
values of \((x_2, \dots, x_n) \in \rational^{n - 1}\)
such that \(H(x_i) \leq B\) for all \(i\).
Since \(f\) is monic in \(X_1\),
we see that \(f(X_1, x_2, \dots, x_n)\) has degree \(d_1\),
so it has at most \(d_1\) rational roots.
That is,
for a fixed \((x_2, \dots, x_n)\),
there are at most \(d_1\) values of \(x_1\)
such that \((x_1, \dots, x_n)\) contributes to \(\nrat(f, B)\).
Hence
\[
    \nrat(f, B) \ll_{d, n} B^{2 (n - 2 + 1 / d_1)} (\log\norm{f} + \log B)^7.
\]

\section{Future research directions}
\label{sec:future_research_directions}

Finally,
we state some possible directions for future research.

We do not expect the exponent \(2 n - 3\)
in Theorem~\ref{thm:rational_points_hypersurface} to be optimal.
One possible direction for research would be
to either improve the exponent in the bound
or prove a matching lower bound for \(\nrat(f, B)\).
We believe that \(2 (n - 2 + 1 / d_1)\),
which is the exponent in Theorem~\ref{thm:rational_points_hypersurface_monic},
is the optimal exponent.

In Theorem~\ref{thm:rational_points_hypersurface},
we assume that \(f\) has degree at least \(2\) in \(X_1\),
and in Theorem~\ref{thm:rational_points_hypersurface_monic},
we assume that \(f\) is monic in \(X_1\).
We do not expect these assumptions to be necessary
for these bounds to hold,
so another possible direction for research
would be to weaken these assumptions.

In Theorems~\ref{thm:rational_points_hypersurface},
\ref{thm:rational_points_hypersurface_monic}
and \ref{thm:hilbert_irreducibility_non_uniform},
one could make the dependence on \(d\) explicit,
improve the dependence on \(\log\norm{f}\)
or improve the exponent of \(\log B\).

Another possible direction for research
would be to generalise the main results to varieties or to global fields.

We considered a quantitative form of Hilbert's irreducibility theorem,
where for an irreducible polynomial
\(f \in \integer[T_1, \dots, T_s, Y_1, \dots, Y_r]\),
we bounded the number of \((t_1, \dots, t_s) \in \rational^s\)
of bounded height
such that \(f(t_1, \dots, t_s, Y_1, \dots, Y_r)\) is irreducible.
Another form of Hilbert's irreducibility theorem concerns the Galois group
of the specialisations.
One possible direction for research
would be to obtain quantitative bounds in this setting.
More precisely,
given an irreducible polynomial \(f \in \integer[T_1, \dots, T_s, Y]\)
and group \(G\),
one can ask for a bound on the number of \((t_1, \dots, t_s) \in \rational^s\)
of bounded height such that \(f(t_1, \dots, t_s, Y)\) has Galois group \(G\).

\printbibliography
\end{document}